\documentclass{amsart}
\usepackage{amssymb}
\usepackage{amsbsy, amsthm, amsmath,amstext, amsopn}
\usepackage[all]{xy}
\usepackage{amsfonts}
\usepackage{amscd}
\usepackage{stmaryrd}
\usepackage{color}
\usepackage[numbers, square, sort, comma]{natbib}
\usepackage{pdfpages}

\theoremstyle{definition}

\newtheorem{thm}{Theorem} [section]
\newtheorem{lemma}[thm]{Lemma}

\newtheorem{corollary}[thm]{Corollary}
\newtheorem{prop}[thm]{Proposition}

\newtheorem{defn}[thm]{Definition}

\newtheorem{example}[thm]{Example}

\newtheorem{remark}[thm]{Remark}

\begin{document}

\numberwithin{equation}{section}

\newcommand{\hs}{\mbox{\hspace{.4em}}}
\newcommand{\ds}{\displaystyle}
\newcommand{\bd}{\begin{displaymath}}
\newcommand{\ed}{\end{displaymath}}
\newcommand{\bcd}{\begin{CD}}
\newcommand{\ecd}{\end{CD}}

\newcommand{\on}{\operatorname}
\newcommand{\proj}{\operatorname{Proj}}
\newcommand{\bproj}{\underline{\operatorname{Proj}}}

\newcommand{\spec}{\operatorname{Spec}}
\newcommand{\Spec}{\operatorname{Spec}}
\newcommand{\bspec}{\underline{\operatorname{Spec}}}
\newcommand{\pline}{{\mathbf P} ^1}
\newcommand{\aline}{{\mathbf A} ^1}
\newcommand{\pplane}{{\mathbf P}^2}
\newcommand{\aplane}{{\mathbf A}^2}
\newcommand{\coker}{{\operatorname{coker}}}
\newcommand{\ldb}{[[}
\newcommand{\rdb}{]]}

\newcommand{\Sym}{\operatorname{Sym}^{\bullet}}
\newcommand{\Symp}{\operatorname{Sym}}
\newcommand{\Pic}{\bf{Pic}}
\newcommand{\Aut}{\operatorname{Aut}}
\newcommand{\PAut}{\operatorname{PAut}}

\newcommand{\too}{\twoheadrightarrow}
\newcommand{\C}{{\mathbf C}}
\newcommand{\Z}{{\mathbf Z}}
\newcommand{\Q}{{\mathbf Q}}
\newcommand{\Cx}{{\mathbf C}^{\times}}
\newcommand{\Cbar}{\overline{\C}}
\newcommand{\Cxbar}{\overline{\Cx}}
\newcommand{\cA}{{\mathcal A}}
\newcommand{\cS}{{\mathcal S}}
\newcommand{\cV}{{\mathcal V}}
\newcommand{\cM}{{\mathcal M}}
\newcommand{\bA}{{\mathbf A}}
\newcommand{\cB}{{\mathcal B}}
\newcommand{\cC}{{\mathcal C}}
\newcommand{\cD}{{\mathcal D}}
\newcommand{\D}{{\mathcal D}}
\newcommand{\cs}{{\mathbf C} ^*}
\newcommand{\boldc}{{\mathbf C}}
\newcommand{\cE}{{\mathcal E}}
\newcommand{\cF}{{\mathcal F}}
\newcommand{\bF}{{\mathbf F}}
\newcommand{\cG}{{\mathcal G}}
\newcommand{\G}{{\mathbb G}}
\newcommand{\cH}{{\mathcal H}}
\newcommand{\CI}{{\mathcal I}}
\newcommand{\cJ}{{\mathcal J}}
\newcommand{\cK}{{\mathcal K}}
\newcommand{\cL}{{\mathcal L}}
\newcommand{\baL}{{\overline{\mathcal L}}}

\newcommand{\Mf}{{\mathfrak M}}
\newcommand{\bM}{{\mathbf M}}
\newcommand{\bm}{{\mathbf m}}
\newcommand{\cN}{{\mathcal N}}
\newcommand{\theo}{\mathcal{O}}
\newcommand{\cP}{{\mathcal P}}
\newcommand{\cR}{{\mathcal R}}
\newcommand{\Pp}{{\mathbb P}}
\newcommand{\boldp}{{\mathbf P}}
\newcommand{\boldq}{{\mathbf Q}}
\newcommand{\bbL}{{\mathbf L}}
\newcommand{\cQ}{{\mathcal Q}}
\newcommand{\cO}{{\mathcal O}}
\newcommand{\Oo}{{\mathcal O}}
\newcommand{\cY}{{\mathcal Y}}
\newcommand{\OX}{{\Oo_X}}
\newcommand{\OY}{{\Oo_Y}}
\newcommand{\otY}{{\underset{\OY}{\ot}}}
\newcommand{\otX}{{\underset{\OX}{\ot}}}
\newcommand{\cU}{{\mathcal U}}\newcommand{\cX}{{\mathcal X}}
\newcommand{\cW}{{\mathcal W}}
\newcommand{\boldz}{{\mathbf Z}}
\newcommand{\qgr}{\operatorname{q-gr}}
\newcommand{\gr}{\operatorname{gr}}
\newcommand{\rk}{\operatorname{rk}}
\newcommand{\SH}{{\underline{\operatorname{Sh}}}}
\newcommand{\End}{\operatorname{End}}
\newcommand{\uEnd}{\underline{\operatorname{End}}}
\newcommand{\Hom}{\operatorname{Hom}}
\newcommand{\uHom}{\underline{\operatorname{Hom}}}
\newcommand{\uHomY}{\uHom_{\OY}}
\newcommand{\uHomX}{\uHom_{\OX}}
\newcommand{\Ext}{\operatorname{Ext}}
\newcommand{\bExt}{\operatorname{\bf{Ext}}}
\newcommand{\Tor}{\operatorname{Tor}}

\newcommand{\inv}{^{-1}}
\newcommand{\airtilde}{\widetilde{\hspace{.5em}}}
\newcommand{\airhat}{\widehat{\hspace{.5em}}}
\newcommand{\nt}{^{\circ}}
\newcommand{\del}{\partial}

\newcommand{\supp}{\operatorname{supp}}
\newcommand{\ssupp}{\mathit{ss}}
\newcommand{\GK}{\operatorname{GK-dim}}
\newcommand{\hd}{\operatorname{hd}}
\newcommand{\id}{\operatorname{id}}
\newcommand{\res}{\operatorname{res}}
\newcommand{\lrar}{\leadsto}
\newcommand{\im}{\operatorname{Im}}
\newcommand{\HH}{\operatorname{H}}
\newcommand{\TF}{\operatorname{TF}}
\newcommand{\Bun}{\operatorname{Bun}}

\newcommand{\F}{\mathcal{F}}
\newcommand{\Ff}{\mathbb{F}}
\newcommand{\nthord}{^{(n)}}
\newcommand{\Gr}{{\mathfrak{Gr}}}

\newcommand{\Fr}{\operatorname{Fr}}
\newcommand{\GL}{\operatorname{GL}}
\newcommand{\gl}{\mathfrak{gl}}
\newcommand{\SL}{\operatorname{SL}}
\newcommand{\ff}{\footnote}
\newcommand{\ot}{\otimes}
\def\Ext{\operatorname {Ext}}
\def\Hom{\operatorname {Hom}}
\def\Ind{\operatorname {Ind}}
\def\bbZ{{\mathbb Z}}

\newcommand{\nc}{\newcommand}
\nc{\ol}{\overline} \nc{\cont}{\on{cont}} \nc{\rmod}{\on{mod}}
\nc{\Mtil}{\widetilde{M}} \nc{\wb}{\overline} 
\nc{\wh}{\widehat}  \nc{\mc}{\mathcal}
\nc{\mbb}{\mathbb}  \nc{\K}{{\mc K}} \nc{\Kx}{{\mc K}^{\times}}
\nc{\Ox}{{\mc O}^{\times}} \nc{\unit}{{\bf \on{unit}}}
\nc{\boxt}{\boxtimes} \nc{\xarr}{\stackrel{\rightarrow}{x}}

\nc{\Ga}{\G_a}
 \nc{\PGL}{{\on{PGL}}}
 \nc{\PU}{{\on{PU}}}

\nc{\h}{{\mathfrak h}} \nc{\kk}{{\mathfrak k}}
 \nc{\Gm}{\G_m}
\nc{\Gabar}{\wb{\G}_a} \nc{\Gmbar}{\wb{\G}_m} \nc{\Gv}{G^\vee}
\nc{\Tv}{T^\vee} \nc{\Bv}{B^\vee} \nc{\g}{{\mathfrak g}}
\nc{\gv}{{\mathfrak g}^\vee} \nc{\BRGv}{\on{Rep}\Gv}
\nc{\BRTv}{\on{Rep}T^\vee}
 \nc{\Flv}{{\mathcal B}^\vee}
 \nc{\TFlv}{T^*\Flv}
 \nc{\Fl}{{\mathfrak Fl}}
\nc{\BRR}{{\mathcal R}} \nc{\Nv}{{\mathcal{N}}^\vee}
\nc{\St}{{\mathcal St}} \nc{\ST}{{\underline{\mathcal St}}}
\nc{\Hec}{{\bf{\mathcal H}}} \nc{\Hecblock}{{\bf{\mathcal
H_{\alpha,\beta}}}} \nc{\dualHec}{{\bf{\mathcal H^\vee}}}
\nc{\dualHecblock}{{\bf{\mathcal H^\vee_{\alpha,\beta}}}}
\newcommand{\ramBun}{{\bf{Bun}}}
\newcommand{\ramBuno}{\ramBun^{\circ}}

\nc{\Buntheta}{{\bf Bun}_{\theta}} \nc{\Bunthetao}{{\bf
Bun}_{\theta}^{\circ}} \nc{\BunGR}{{\bf Bun}_{G_\BR}}
\nc{\BunGRo}{{\bf Bun}_{G_\BR}^{\circ}}
\nc{\HC}{{\mathcal{HC}}}
\nc{\risom}{\stackrel{\sim}{\to}} \nc{\Hv}{{H^\vee}}
\nc{\bS}{{\mathbf S}}
\def\BRep{\operatorname {Rep}}
\def\Conn{\operatorname {Conn}}

\nc{\Vect}{{\operatorname{Vect}}}
\nc{\Hecke}{{\operatorname{Hecke}}}

\newcommand{\ZZ}{{Z_{\bullet}}}
\nc{\HZ}{{\mc H}\ZZ} \nc{\eps}{\epsilon}

\nc{\CN}{\mathcal N} \nc{\BA}{\mathbb A}

\nc{\ul}{\underline}

\nc{\bn}{\mathbf n} \nc{\Sets}{{\mathit{Sets}}} \nc{\Top}{{\on{Top}}}
\nc{\IntHom}{{\mathcal Hom}}

\nc{\Simp}{{\mathbf \Delta}} \nc{\Simpop}{{\mathbf\Delta^\circ}}

\nc{\Cyc}{{\mathbf \Lambda}} \nc{\Cycop}{{\mathbf\Lambda^\circ}}

\nc{\Mon}{{\mathbf \Lambda^{mon}}}
\nc{\Monop}{{(\mathbf\Lambda^{mon})\circ}}

\nc{\Aff}{{\on{Aff}}} \nc{\Sch}{{\on{Sch}}}

\nc{\bul}{\bullet}
\nc{\module}{{\operatorname{-mod}}}

\nc{\dstack}{{\mathcal D}}

\nc{\BL}{{\mathbb L}}

\nc{\BD}{{\mathbb D}}

\nc{\BR}{{\mathbb R}}

\nc{\BT}{{\mathbb T}}

\nc{\SCA}{{\mc{SCA}}}
\nc{\DGA}{{\mc DGA}}

\nc{\DSt}{{DSt}}

\nc{\lotimes}{{\otimes}^{\mathbf L}}

\nc{\bs}{\backslash}

\nc{\Lhat}{\widehat{\mc L}}

\newcommand{\Coh}{\on{Coh}}

\nc{\QCoh}{QC}
\nc{\QC}{QC}
\nc{\Perf}{\on{Perf}}
\nc{\Cat}{{\on{Cat}}}
\nc{\dgCat}{{\on{dgCat}}}
\nc{\bLa}{{\mathbf \Lambda}}

\nc{\BRHom}{\mathbf{R}\hspace{-0.15em}\on{Hom}}
\nc{\BREnd}{\mathbf{R}\hspace{-0.15em}\on{End}}
\nc{\colim}{\on{colim}}
\nc{\oo}{\infty}
\nc{\Mod}{\on{Mod} }

\nc\fh{\mathfrak h}
\nc\al{\alpha}
\nc\la{\alpha}
\nc\BGB{B\bs G/B}
\nc\QCb{QC^\flat}
\nc\qc{\on{QC}}

\def\w{\wedge}
\nc{\vareps}{\varepsilon}

\nc{\fg}{\mathfrak g}

\nc{\Map}{\on{Map}} \nc{\fX}{\mathfrak X}

\nc{\ch}{\check}
\nc{\fb}{\mathfrak b} \nc{\fu}{\mathfrak u} \nc{\st}{{st}}
\nc{\fU}{\mathfrak U}
\nc{\fZ}{\mathfrak Z}

 \nc\fc{\mathfrak c}
 \nc\fs{\mathfrak s}

\nc\fk{\mathfrak k} \nc\fp{\mathfrak p}

\nc{\BRP}{\mathbf{RP}} \nc{\rigid}{\text{rigid}}
\nc{\glob}{\text{glob}}

\nc{\cI}{\mathcal I}

\nc{\La}{\mathcal L}

\nc{\quot}{/\hspace{-.25em}/}

\nc\aff{\it{aff}}
\nc\BS{\mathbb S}

\nc\Loc{{\mc Loc}}
\nc\Tr{{\on{Tr}}}
\nc\Ch{{\mc Ch}}

\nc\ftr{{\mathfrak {tr}}}
\nc\fM{\mathfrak M}

\nc\Id{\operatorname{Id}}

\nc\bimod{\on{-bimod}}

\nc\ev{\operatorname{ev}}
\nc\coev{\operatorname{coev}}

\nc\pair{\operatorname{pair}}
\nc\kernel{\operatorname{kernel}}

\nc\Alg{\operatorname{Alg}}

\nc\init{\emptyset_{\text{\em init}}}
\nc\term{\emptyset_{\text{\em term}}}

\nc\Ev{\on{Ev}}
\nc\Coev{\on{Coev}}

\nc\es{\emptyset}
\nc\m{\text{\it min}}
\nc\M{\text{\it max}}
\nc\cross{\text{\it cr}}
\nc\tr{\on{tr}}

\nc\perf{\on{-perf}}
\nc\inthom{\mathcal Hom}
\nc\intend{\mathcal End}

\newcommand{\Sh}{\mathit{Sh}}

\nc{\Comod}{\on{Comod}}
\nc{\cZ}{\mathcal Z}

\def\interiorsymbol {\on{int}}

\nc\frakf{\mathfrak f}
\nc\fraki{\mathfrak i}
\nc\frakj{\mathfrak j}
\nc\bP{\mathbb P}
\nc\stab{st}
\nc\Stab{St}

\nc\fN{\mathfrak N}
\nc\fT{\mathfrak T}
\nc\fV{\mathfrak V}

\nc\Ob{\on{Ob}}

\nc\fC{\mathfrak C}
\nc\Fun{\on{Fun}}

\nc\Null{\on{Null}}

\nc\BC{\mathbb C}

\nc\loc{\on{Loc}}

\nc\hra{\hookrightarrow}
\nc\fL{\mathfrak L}
\nc\R{\mathbb R}
\nc\CE{\mathcal E}

\nc\sK{\mathsf K}
\nc\sC{\mathsf C}

\nc\Cone{\mathit Cone}

\nc\fY{\mathfrak Y}
\nc\fe{\mathfrak e}
\nc\ft{\mathfrak t}

\nc\wt{\widetilde}
\nc\inj{\mathit{inj}}
\nc\surj{\mathit{surj}}

\nc\Path{\mathit{Path}}
\nc\Set{\mathit{Set}}
\nc\Fin{\mathit{Fin}}

\nc\cyc{\mathit{cyc}}

\nc\per{\mathit{per}}

\nc\sym{\mathit{symp}}
\nc\con{\mathit{cont}}
\nc\gen{\mathit{gen}}
\nc\str{\mathit{str}}
\nc\rsdl{\mathit{res}}
\nc\rel{\mathit{rel}}
\nc\pt{\mathit{pt}}
\nc\naive{\mathit{nv}}
\nc\forget{\mathit{For}}

\nc\sH{\mathsf H}
\nc\sW{\mathsf W}
\nc\sE{\mathsf E}
\nc\sP{\mathsf P}
\nc\sB{\mathsf B}
\nc\sR{\mathsf R}
\nc\sQ{\mathsf Q}
\nc\sL{\mathsf L}
\nc\si{\mathsf i}
\nc\sj{\mathsf j}
\nc\sk{\mathsf k}
\nc\sq{\mathsf q}
\nc\sr{\mathsf r}

\nc\sS{\mathsf S}
\nc\fH{\mathfrak H}
\nc\fP{\mathfrak P}
\nc\fq{\mathfrak q}
\nc\fW{\mathfrak W}
\nc\fE{\mathfrak E}
\nc\sx{\mathsf x}
\nc\sy{\mathsf y}

\nc\ord{\mathit{ord}}

\nc\sm{\mathit{sm}}
\nc\sing{\mathit{sing}}
\nc\internal{\mathit{int}}
\nc\midpt{\mathit{mid}}
\nc\yes{\mathit{yes}}
\nc\no{\mathit{no}}
\nc\trees{\mathit{Trees}}

\nc\rhu{\rightharpoonup}
\nc\dirT{\mathcal T}
\nc\link{\mathit{link}}


\title[Arboreal singularities]{Arboreal singularities
}

\author{David Nadler}
\address{Department of Mathematics\\University of California, Berkeley\\Berkeley, CA  94720-3840}
\email{nadler@math.berkeley.edu}

\begin{abstract}
We introduce a class of combinatorial singularities of Lagrangian skeleta of  symplectic manifolds. The link of each singularity is   a finite regular cell complex homotopy equivalent to a bouquet of spheres. It is determined by its face poset which
 is naturally constructed starting from  a tree (nonempty finite acyclic graph). The choice of a root vertex of the tree leads to a natural  front projection of the singularity along with an   orientation of the edges of the tree.
Microlocal sheaves along the singularity, calculated via the front projection, are equivalent to modules over the quiver given by the directed tree.
\end{abstract}

\maketitle


\tableofcontents


\section{Introduction}

This paper is part of a project devoted to combinatorial models of symplectic topology,
in particular of singular Lagrangian skeleta.
After a  summary of our main results immediately below, we  discuss in Section~\ref{s motivation}  the subsequent development of the theory
~\cite{Nexp},  which proves a refined version of a conjecture of
Kontsevich~\cite{kont}, and has applications to mirror symmetry~\cite{N3dlg, Nwms}. 
On the one hand, this paper introduces the main objects and core calculations and is essential to what follows. 
On the other hand, by design, this paper can be read independently of further developments and with a minimal amount of geometric background:
its constructions are of a combinatorial nature, and its results give  elementary  realizations  of microlocal invariants. 
Its main results include calculations of microlocal sheaves where the answer can be viewed as  an appealing alternative  to  traditional technical definitions.  Low-dimensional examples of the main objects also arise naturally in recent  advances in Legendrian knot theory found in~\cite{STZ,
NRSSZ} and related work.  

\subsection{Summary}\label{s summary}
We will 
 introduce a class of combinatorial singularities, first as coarse topological spaces,
 then naturally embedded as Legendrian singularities.

Our starting point
is a {\em tree} $T$ in the sense of a nonempty finite connected acyclic graph.

\begin{figure}[h!]
\includegraphics[scale=0.5]{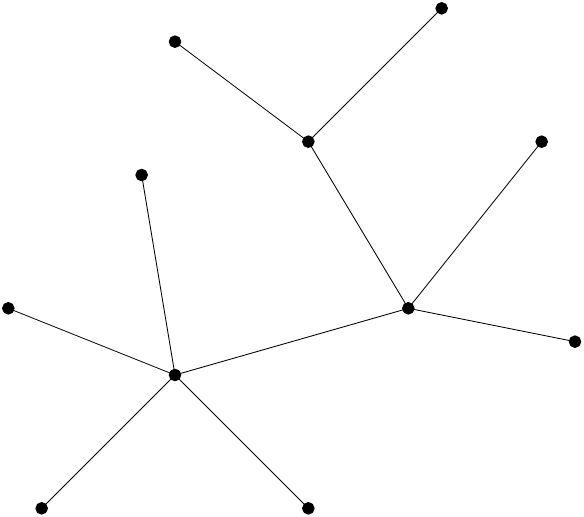}
\caption{Example of a tree $T$.}
\end{figure}

To each tree $T$, we associate a stratified space $\sL_T$  called an {\em arboreal singularity}. It is of pure dimension $|T| - 1$ where we write $|T|$ for  the number of vertices of $T$.
It comes equipped with a compatible   metric and    contracting $\R_{>0}$-action with a single fixed point.  We refer to the compact subspace $\sL_T^\link \subset \sL_T$ of points unit distance from the fixed point
as the {\em arboreal link}. The   $\R_{>0}$-action provides   a canonical identification
$$
\xymatrix{
\sL_T \simeq \Cone(\sL_T^\link)
}
$$
so that one can regard the arboreal singularity $\sL_T$ and arboreal link $\sL_T^\link$ as  respective local models for a normal slice and normal link to a stratum in a stratified space. It follows easily from the constructions that the arboreal link $\sL_T^\link$ is homotopy equivalent to a bouquet of $|T|$ spheres each of dimension $|T|-1$.

As a stratified space, the arboreal link $\sL_T^\link$, and hence the arboreal singularity $\sL_T$ as well, admits a simple  combinatorial description. 
 To each tree $T$, there is a natural finite poset $\fP_T$ whose elements are correspondences of trees
$$
\xymatrix{
\fp=(R &  \ar@{->>}[l]_-q  S \ar@{^(->}[r]^-i & T)
}
$$
where $i$ is the inclusion of  a subtree and $q$ is a quotient  of trees. More precisely, the tree $S$ is the full subgraph (or vertex-induced subgraph) on a subset of vertices of $T$; the tree $R$ results from contracting a subset of edges of $S$. 
Two such correspondences   
$$
\xymatrix{
\fp=(R &  \ar@{->>}[l]_-q  S \ar@{^(->}[r]^-i & T)
&
\fp'=(R' &  \ar@{->>}[l]_-{q'}  S' \ar@{^(->}[r]^-{i'} & T')
}
$$
satisfy $\fp\geq  \fp'$ 
if there is another  correspondence of the same form
$$
\xymatrix{
\fq=(R &  \ar@{->>}[l]  Q \ar@{^(->}[r] & R')
}
$$
 such that $\fp = \fq\circ\fp'$. In particular, the poset $\fP_T$ contains a unique minimum representing the identity correspondence
 $$
\xymatrix{
\fp_0=(T &  \ar@{->>}[l]_-=  T \ar@{^(->}[r]^-= & T)
}
$$

Recall that a {\em finite regular cell complex} is a Hausdorff space $X$
with a finite collection of closed cells $c_i \subset X$ whose interiors $c_i^\circ \subset c_i$ provide a partition of $X$ and boundaries $\partial c_i \subset X$
are unions of cells. A finite regular cell complex $X$ has the {\em intersection property} if the intersection of any two cells $c_i, c_j\subset X$ is 
either another cell or empty. The {\em face poset} of a finite regular cell complex $X$ is the poset with elements the cells of $X$ with
relation $c_i\leq c_j$ whenever $c_i \subset  c_j$. The {\em order complex} of a poset is the natural simplicial complex with simplices the finite totally-ordered chains of the poset. 
(Useful references include~ \cite{bb, bjorner, wachs}.)

As topological spaces, arboreal singularities take the following simple combinatorial form.
If we were only interested in their topology, we could take the below description as definition.
Instead, we will approach them with a geometric  construction that leads to their natural realization as Legendrian singularities.

\begin{thm}
Let $T$ be a tree.

The arboreal link $\sL_T^\link$ is a finite regular cell complex, with the intersection property, with face poset $
 \fP_T\setminus \{\fp_0\}$, and  thus  homeomorphic to the order complex of $\fP_T\setminus \{\fp_0\}$.
\end{thm}

%
%

\begin{remark}
It follows  from the theorem and  the poset structure on $\fP_T$   that the normal slice to the stratum  $\sL_T(\fp) \subset \sL_T$ indexed by a partition
$$
\xymatrix{
\fp=(R &  \ar@{->>}[l]_-q  S \ar@{^(->}[r]^-i & T) 
}
$$
is homeomorphic to the arboreal singularity $\sL_R$.
\end{remark}


%

\begin{example}
Let us highlight the simplest class of trees.

When $T$ consists of a single vertex, $\sL_T$ is a single point.

When $T$ consists of two vertices $v_1, v_2$ (necessarily connected by an edge), $\sL_T$ is the local trivalent graph given by the cone over the three distinct points $\sL_T^\link$ representing the three correspondences
 $$
\xymatrix{
(\{v_1\} &  \ar@{->>}[l]_-= \{v_1\} \ar@{^(->}[r] & T)
&
(\{v_2\} &  \ar@{->>}[l]_-= \{v_2\} \ar@{^(->}[r] & T)
&
(\{v\} &  \ar@{->>}[l] T \ar@{^(->}[r]^-= & T)
}
$$

More generally, consider  the class of  $A_n$-trees $T_n$ consisting of $n$ vertices 
connected by $n-1$ successive edges. 
The associated arboreal singularity $\sL_{T_n}$ 
admits an identification with the cone of the  $(n-2)$-skeleton of the $n$-simplex
$$
\xymatrix{
\sL_{T_n} \simeq \Cone(sk_{n-2} \Delta^n)
}$$
or in a dual realization, the $(n-1)$-skeleton of the polar fan of the $n$-simplex.
This space
arises in many places (all intimately related to symplectic topology):

\begin{enumerate}
\item  a  tropical hyperplane in $n$-dimensional tropical projective space (\cite{ad, ds, ss}),\footnote
{
We thank E. Zaslow for pointing this out to us, and D. Auroux for noting this perspective appears
in Kontsevich's expectations~\cite{kont}. No doubt it holds significance for mirror symmetry.}

\item  the universal  planar tree  over the $(n-2)$-dimensional associahedron  $K_{n-2}$ (\cite{Stasheff, loday, Seidel}),

\item in  geometric realizations of Waldhausen's $S$-construction in $K$-theory (\cite{DK, DKcyc, Ncyc}).
\end{enumerate}

\end{example}

\begin{figure}[h!]
  \begin{center}  
\includegraphics[width=20em,height=12em]{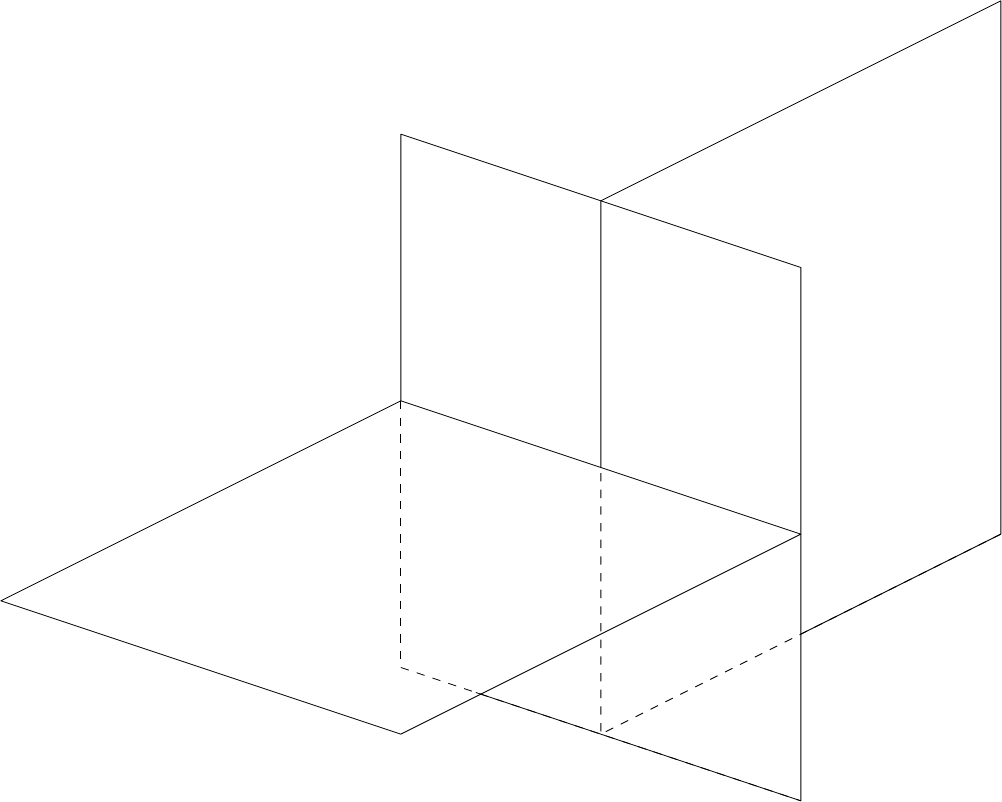}
   \caption{The $A_3$  arboreal singularity.}
   \label{fig: a3}
\end{center}
\end{figure}

\begin{example}
The first  example beyond $A_n$-trees is that of the $D_4$-tree with a central vertex connected to three other vertices.
The corresponding arboreal singularity is the union of a Euclidean space $\R^3$ and three closed Euclidean halfspaces $\R_{\geq 0} \times \R^2$  each glued along its boundary $\R^2 = \partial(\R_{\geq 0} \times \R^2)$ to the Euclidean space $\R^3$ along a distinct coordinate hyperplane $\R^2 \subset \R^3$.
\end{example}

Arboreal singularities  offer a natural generalization of the above 
singularities associated to $A_n$-trees.
We will next explain their appearance as Legendrian singularities  whose front projections
are particularly simple cooriented singular hypersurfaces. (Then in Section~\ref{s motivation} below
we will discuss their related appearance  as Lagrangian singularities.)

To this end, our refined starting point is a {\em rooted tree} $\dirT= (T, \rho)$ in the sense of a tree $T$ together with a distinguished vertex $\rho$ called the root vertex.


The set of vertices $V(T)$   naturally forms a poset with the root vertex $\rho \in V(T)$ the  unique minimum
and in general $\alpha< \beta\in V(T)$ if the former is nearer to $\rho$ than the latter. 

Let us write $\R^\dirT$ for the Euclidean space 
of real tuples
$
\{x_\gamma\}
$
indexed by vertices $\gamma \in V(T)$.
Let us  write  $S^*\R^\dirT$
for its spherically projectivized
cotangent bundle, or equivalently  unit cosphere bundle.
 Points of  $S^* \R^\dirT$ are pairs $(x, [v])$ where $x\in \R^\dirT$
and
 $[v]$ is the positive ray,  or equivalently  unit covector,  in the direction of  $v \not = 0 \in T_x^* \R^\dirT$.
Recall that $S^*\R^\dirT$ is naturally a
 cooriented contact manifold.

To each rooted tree $\dirT= (T, \rho)$, we associate a singular hypersurface $\sH_\dirT\subset \R^\dirT$ called an {\em  arboreal hypersurface}. On the one hand, the arboreal hypersurface $\sH_\dirT\subset \R^\dirT$ admits a homeomorphism with the rectilinear hypersurface
defined by coordinate equalities and inequalities
$$
\xymatrix{
H_\dirT \simeq \bigcup_{\alpha\in V(T)} \{x_\alpha = 0,  x_\beta > 0 \text{ for all } \beta<  \alpha \}   \subset \R^\dirT
}$$
On the other hand, the arboreal hypersurface $\sH_\dirT\subset \R^\dirT$ is in {\em good position} in the sense that it has finitely many normal Gauss directions even across its singularities. Thus 
it defines a {\em conormal Legendrian} $\cL^*_{\sH_\dirT} \subset S^*\R^\dirT$ whose front projection provides a finite surjection 
$$
\xymatrix{
\cL^*_{\sH_\dirT} \ar@{->>}[r] & \sH_\dirT
}$$

The following shows that the arboreal singularity $\sL_T$ associated to a tree $T$ naturally arises as a  Legendrian singularity. The choice of the root vertex $\rho\in V(T)$ plays the role  of a polarization enabling this presentation.

\begin{thm}
Let $\dirT = (T, \rho)$ be a rooted tree. 

The conormal Legendrian $\cL^*_{\sH_\dirT} \subset S^*\R^\dirT$  of the arboreal hypersurface $\sH_\dirT \subset \R^\dirT$ is homeomorphic to the  arboreal singularity $\sL_T$.

\end{thm}

\begin{example}
An instructive example is that of the $A_3$-tree $T_3$
with its two possible inequivalent rooted structures. On the one hand, we could take one of the two end vertices as root vertex to obtain a  rooted tree. 
On the other hand, we could take the middle vertex as root vertex to obtain a rooted tree. 
 The resulting 
arboreal hypersurfaces are quite different though
their conormal Legendrians are homeomorphic. 
\end{example}

With the theorem in mind, we will write $\sL_\dirT\subset S^*\R^\dirT$ in place of $\cL^*_{\sH_\dirT} \subset S^*\R^\dirT$,
using the subscript $\dirT$ as opposed to $T$ to emphasize the dependence of the  embedding 
on the poset structure.

We will next  calculate the categorical quantization of the Legendrian singularity $\sL_\dirT\subset S^*\R^\dirT$  in the form of microlocal sheaves supported along it. (We recommend the comprehensive book~\cite{KS} for the general notions that appear
in what follows, along with~\cite{KellerICM} and the references therein  for working in a differential graded setting.)

Fix once and for all a field $k$, and let $\Sh(\R^\dirT)$ denote the dg category of cohomologically constructible complexes of sheaves of $k$-vector spaces on $\R^\dirT$.
Recall that to any object $\cF\in \Sh(\R^\dirT)$, one can associate its singular support $\ssupp(\cF) \subset S^*\R^\dirT$. This is a closed Legendrian subspace recording those codirections in which the propagation of sections of $\cF$ is obstructed. 
In particular, one has the vanishing $\ssupp(\cF) = \emptyset$ if and only if the cohomology sheaves of $\cF$ are locally constant.

Introduce  the dg category  $\Sh_{\sL_T}(\R^\dirT)$  of constructible complexes of sheaves of $k$-vector spaces on $\R^\dirT$
microlocalized along 
$
\sL_\dirT \subset S^*\R^\dirT.
$
Thanks to the simplicity of the situation, we can concretely work with $\Sh_{\sL_\dirT}(\R^\dirT)$ as the full dg  subcategory of $\Sh(\R^\dirT)$ consisting of objects
$\cF \in \Sh(\R^\dirT)$ with the prescribed singular support and vanishing global sections:
\begin{enumerate}
\item $\ssupp(\cF) \subset \sL_\dirT$, and
\item $\Hom_{\Sh(\R^\dirT)}(k_{\R^\dirT}, \cF) \simeq 0$.
\end{enumerate}

Recall that we can regard the set of vertices $V(T)$ of the rooted tree $\dirT = (T, \rho)$ as a poset with the root vertex $\rho\in V(T)$
the unique minimum. 
To each non-root vertex $\alpha \not = \rho \in V(T)$ there is a unique parent vertex $\hat \alpha\in V(T)$ such that $\alpha> \hat \alpha$ and there are no  vertices strictly between them.

Now let us regard the rooted tree  $\dirT = (T, \rho)$ as a quiver with a unique arrow pointing from each non-root vertex $\alpha \not = \rho \in V(T)$ to its parent vertex
$\hat \alpha\in V(T)$. Symbolically, we replace the relation $\alpha >\hat\alpha$ with the relation $\alpha \to \hat \alpha$.

 Let $ \Mod(\dirT)$ denote the dg derived category of finite-dimensional complexes of modules over  $\dirT$ regarded as a quiver.
 Objects assign to each vertex $\alpha \in V(T)$ a finite-dimensional complex of $k$-vector spaces $M(\alpha)$, and to each arrow $\alpha \to \hat \alpha$ a degree zero chain map $m_\alpha:M(\alpha) \to M({\hat \alpha})$.

\begin{remark}
 Let us point out two natural generating collections for $\Mod(\dirT)$. There are the simple modules $S_\alpha \in \Mod(\dirT)$ that assign 
 $$
S_\alpha(\beta) = \left\{\begin{array}{cl}
k  &  \text{ when } \beta = \alpha\\
0 & \text{ when } \beta\not =\alpha
\end{array}\right.
$$
with all maps $m_\beta:S_\alpha(\beta) \to S_\alpha(\hat \beta)$ necessarily zero.
 There are also the projective modules $P_\alpha \in \Mod(\dirT)$ that assign 
 $$
P_\alpha(\beta) = \left\{\begin{array}{cl}
k  &  \text{ when } \beta \leq \alpha\\
0 & \text{ when } \beta > \alpha
\end{array}\right.
$$ 
with the  maps $m_\beta:P_\alpha(\beta) \to P_\alpha(\hat \beta)$ the identity isomorphism whenever both domain and range are nonzero.

\end{remark}

The categorical quantization of the Legendrian singularity $\sL_\dirT\subset S^*\R^\dirT$ admits the following simple description.

\begin{thm}\label{thm:introquant}
Let $\dirT = (T, \rho)$ be a rooted tree.

The dg category  $\Sh_{\sL_\dirT}(\R^\dirT)$ of constructible complexes microlocalized along 
 $\sL_\dirT\subset S^*\R^\dirT$
 is canonically equivalent to the dg category of modules $\Mod(\dirT)$.
 \end{thm}
 
 \begin{remark}
The dg category $\Mod(\dirT)$ is non-canonically independent of the choice of root vertex and resulting quiver structure. Namely, for a different choice of orientations of arrows, reflection functors~\cite{bgp} provide equivalences between the corresponding module categories. Thus the dg category of  microlocal sheaves along the arboreal singularity $\sL_T$ is non-canonically independent  of its presentation as the conormal Legendrian to a particular arboreal hypersurface. 
\end{remark}

 It is  also possible to describe the natural microlocal restriction functors.
Recall that the normal slice to the stratum  $\sL_\dirT(\fp) \subset \sL_\dirT$ indexed by a partition
$$
\xymatrix{
\fp=(R &  \ar@{->>}[l]_-q  S \ar@{^(->}[r]^-i & T) 
}
$$
is homeomorphic to the arboreal singularity $\sL_R$. Note that the quiver structure on $\dirT$ naturally induces  quiver structures on $S$ and $R$ which we denote by $\mathcal S$ and $\mathcal R$ respectively. Under the equivalence of the theorem, the corresponding microlocal restriction functor is the natural composite quotient functor
$$
\xymatrix{
\Mod(\dirT) \ar@{->>}[r]^-{i^*} & \Mod(\mathcal S) \ar@{->>}[r]^-{q_!} &  \Mod(\mathcal R)
}$$
where $i^*$ kills the projective object $P_\alpha\in\Mod(\dirT)$  attached to  $\alpha \in V(T)$ such that $\alpha\not\in i(V(S))$,
and $q_!$ identifies  the projective objects $P_\alpha, P_\beta\in\Mod(\mathcal S)$  attached to $\alpha, \beta\in V(S)$ such that
$q(\alpha) = q(\beta) \in V(R)$.

\subsection{Motivation}\label{s motivation}
We  briefly discuss here the role of this paper in a broader undertaking.
The definitions and discussion of this section will not be used in the rest of the paper.

 Our primary aim is to construct a
 combinatorial model and computational tool for  the  ``quantum category" of $A$-branes mathematically captured by the
  Kashiwara-Schapira~\cite{KS}  theory of
  microlocal sheaves (topology), the Floer-Fukaya-Seidel theory of wrapped and infinitesimal  Fukaya categories (analysis), and the theory of holonomic modules 
over deformation quantizations (algebra),
 exemplified by $\cD$-modules~\cite{bernstein}. 
In parallel with the cohomology of manifolds, where one has  singular  complexes (topology), Morse and Hodge theory (analysis),  and   de Rham complexes (algebra), 
we seek a parallel to simplicial complexes (combinatorics)  in the study of the  intersection theory of  Lagrangians in  symplectic manifolds. The arboreal singularities of the current paper provide a local model for realizing such a combinatorial model.

To explain this further, let us introduce some basic constructions and useful terminology.

Let $N$ be a co-oriented  contact $2n+1$-dimensional  manifold with contact field $\xi = \ker(\lambda)$ defined by a contact form
$\lambda$.
By a Legendrian subvariety $L \subset N$, 
 we will mean  a closed $n$-dimensional Whitney stratified subspace 
 (satisfying some mild additional properties spelled out in~\cite{Nexp})
 such that $\xi|_Y = 0$, for any submanifold $Y\subset N$ contained within $L$. 
By a Legendrian singularity centered at a submanifold $Y\subset N$, we will mean
the germ along $Y\subset N$ of a Legendrian subvariety containing $Y$ as a closed stratum.

Recall the contact Darboux theorem that any contact manifold $N$ is locally equivalent to 
 the spherical projectivization $S^*\BR^{n+1}$ with its standard contact structure.
 Thus  given a directed tree $\dirT$, with $|\dirT|= n+1 - k$,
 we can view the product $\sL_\dirT \times \BR^k \subset S^*(\BR^\dirT \times \BR^k)$ as a Legendrian singularity within $N$.

For the sake of the current discussion, let us proceed with the following definition which is  more concrete
but less flexible than possible alternatives.

\begin{defn}
A Legendrian subvariety  $L\subset N$ is said to have {\em arboreal singularities}
if its singularity at each of its points is equivalent via a contactomorphism to a  Legendrian singularity of the form $\sL_\dirT \times \BR^k$,
for  a directed tree $\dirT$, with $|\dirT|= n+1 - k$.
\end{defn}

\begin{remark}
If a Legendrian subvariety  $L\subset N$ has arboreal singularities,
then
the dg category of microlocal sheaves on $N$ supported along $L$ can be calculated combinatorially
via Theorem~\ref{thm:introquant} and the functoriality described thereafter.
\end{remark}

In the sequel~\cite{Nexp}, we study arbitrary Legendrian singularities, and prove the following theorem.
The term {\em non-characteristic} in its statement refers to the property that the dg category of microlocal sheaves supported along the Legendrian singularity is unchanged by the deformation. The phrase
{\em degenerate arboreal singularities} refers to a modest variation on arboreal singularities  
discussed in ~\cite{Nexp}. For example, in one dimension,  a trivalent vertex of a graph is an arboreal singularity,
and a univalent vertex is a degenerate arboreal singularity.

\begin{thm}[\cite{Nexp}]\label{thm: nonchar deform}
Any Legendrian singularity admits a non-characteristic deformation to a Legendrian 
subvariety with arboreal and degenerate arboreal singularities.
\end{thm}

Roughly speaking,  to prove the theorem, given a Legendrian singularity, we expand each of its strata into an irreducible component to arrive at a Legendrian  subvariety whose singularities are governed by the combinatorics of the interaction of its irreducible components. 
 With the theorem  in hand,  the calculation of microlocal sheaves  may be performed 
in terms of finite-dimensional modules over trees appealing to the results of the current paper.
One could compare the situation with Morse theory or resolutions with normal crossing divisors in algebraic geometry,
where complicated singularities are reduced to combinatorial assemblages of simple singularities
to calculate invariants.

Finally, to connect with Lagrangian skeleta, 
let $M$ be an exact symplectic $2n$-dimensional manifold, with symplectic form $\omega$ and primitive $\alpha$.
By an exact Lagrangian subvariety $\Lambda \subset M$, 
 we will mean  a closed $n$-dimensional Whitney stratified subspace (satisfying some mild additional properties),
 admitting a continuous function $f:\Lambda\to \BR$,
 such that 
 for any submanifold $Y\subset M$ contained within $\Lambda$, we have
$\omega|_Y = 0$, and $f|_Y$ is differentiable with $d(f|_Y) = \alpha|_Y$, 
   By a Lagrangian singularity centered at a submanifold $Y\subset M$, we will mean
the germ along $Y\subset M$ of a Lagrangian subvariety containing $Y$ as a closed stratum.

Now let us set $N = M \times \BR$ to be the contactification of $M$ with 
 contact field $\xi = \ker(\lambda)$ defined by the contact form
$\lambda = dt + \alpha$, where we write $t$ for a coordinate on $\BR$.
Then any exact Lagrangian subvariety $\Lambda \subset M$, equipped with a primitive $f:\Lambda\to \BR$, lifts to a Legendrian subvariety given by the graph 
$$
\xymatrix{
L_{\Lambda, f} =\{(m, -f(m)) \in \Lambda \times \BR\}\subset N
}
$$
Note that alternative primitives will differ from $f$ by a locally constant function on $\Lambda$,
and hence the corresponding lift will differ from $L_{\Lambda, f}$ by a locally constant translation. 
In this way,  we can embed the study of exact Lagrangian singularities and subvarieties into that of 
Legendrian singularities and subvarieties.
 Notably, we may lift Lagrangian skeleta to Legendrian subvarieties, and then apply the above theory to their singularities.

%
%
\begin{defn}
An exact Lagrangian subvariety
 $\Lambda \subset M$, with primitive $f:\Lambda\to \BR$,
is said to have {\em arboreal singularities} if the Legendrian subvariety $L_{\Lambda, f}\subset N$ 
has arboreal singularities.
\end{defn}

\begin{remark}
Forming the contactification, or  further forming its symplectification, leaves  invariants such as microlocal sheaves 
with prescribed support unchanged.
\end{remark}

\begin{example}
A basic example of Lagrangian skeleta are ribbon graphs in punctured Riemann surfaces. Such a Lagrangian skeleton has arboreal and degenerate arboreal singularities if and only if each of its vertices is trivalent or univalent.
\end{example}

\begin{example}
A common example of a Lagrangian singularity is given by the union of  two smooth Lagrangian submanifolds intersecting transversely. Thus the geometry is locally modeled by $M= T^*\BR^n \simeq \BR^n \times \BR^n$,
$L_1 = \BR^n \times\{0\}$, $L_2 = \{0\} \times \BR^n$, and $L = L_1 \cup L_2$. This is not an arboreal singularity,
but we may apply the above theory to it. Depending on choices in the algorithm underlying Theorem~\ref{thm: nonchar deform}, what results is one of two possible arboreal Lagrangian subvarieties $L_{\pm} \subset M$ given by the respective union 
$$
\xymatrix{
L_{\pm}   = (L_1 \#_\pm L_2) \cup D_\pm^{n}
}
$$
where $ L_1 \#_\pm L_2 \subset M$ is either the positive or negative Lagrangian surgery, 
$D_\pm^n\subset M$ is the respective    vanishing thimble, 
and they meet along  the respective vanishing sphere 
$ S_\pm^{n-1} = \partial D_\pm^n \subset L_1 \#_\pm L_2 $.
Thus the arboreal Lagrangian subvariety $L_{\pm} \subset M$ is smooth except along $S_\pm^{n-1}$
where  its normal geometry is equivalent to the trivalent vertex of  the $A_2$-arboreal singularity.
In the basic case of dimension $n=1$, we recover the two standard trivalent deformations of a four-valent vertex. 
\end{example}

As a sample first application~\cite{N3dlg}, we apply this circle of  ideas  to an important example
in mirror symmetry: the Landau-Ginzburg $A$-model with background $M= \BC^3$
and superpotential $W= z_1z_2z_3$. (Natural generalizations appear in the later work~\cite{Nlg}.)
Due to the fact that the critical locus $\{dW = 0\} \subset M$ is not smooth or proper,  this Landau-Ginzburg $A$-model
is not easily approached with traditional methods. The main theorem of~\cite{N3dlg} is the calculation of microlocal sheaves along the natural singular Lagrangian thimble  $L = \Cone(T^2)\subset M$, and more basically the construction of a deformation of $L$ to a  Lagrangian skeleton with arboreal singularities. The description  obtained is in the form of a quiver with relations, and immediately equivalent
to the $B$-model of the pair-of-pants $\mathbb P^1 \setminus \{0, 1, \infty\}$ as predicted by mirror symmetry. 

In another direction,
we apply the results 
of this paper and the sequel~\cite{Nexp} to prove a duality theorem for microlocal sheaves~\cite{Nwms}.
We introduce wrapped microlocal sheaves, in parallel with wrapped Fukaya categories,
and develop their basic properties.
Most prominently, we show that traditional microlocal sheaves are equivalent to functionals on wrapped microlocal sheaves,
in analogy with the expected relation of infinitesimal to wrapped Fukaya categories,
or the relation of perfect complexes with proper support to coherent sheaves~\cite{BNP}.
To prove this, we reduce the global assertion to a local assertion, then apply 
 the results 
of this paper and the sequel~\cite{Nexp} to reduce to arboreal singularities, and ultimately the fact that finite-dimensional modules over trees form a smooth and proper dg category.
 


\subsection{Acknowledgements}
I thank  D. Auroux, J. Lurie, D. Treumann, L. Williams, and E. Zaslow for their interest, encouragement, and valuable comments. I  also thank D. Ben-Zvi for 
 many inspiring  discussions on a broad range of related and unrelated topics.

I am very grateful to the NSF for the support of grant DMS-1319287.


\section{Arboreal singularities}


\subsection{Gluing construction}

By a {\em graph} $G$, we will mean a set of {\em vertices} $V(G)$ and a set of {\em edges} $E(G)$ satisfying the simplest convention  that $E(G)$ is a
subset of the set of two-element subsets of  $V(G)$. Thus  $E(G)$  records whether pairs of distinct elements  of $V(G)$ are connected by an edge or not. We will write $\{\alpha, \beta\} \in E(G)$ and say that $\alpha, \beta \in V(T)$ are {\em adjacent}  if an edge connects them.

By a  {\em tree} $T$, we will mean a nonempty finite connected acyclic graph. Thus for any  vertices $\alpha, \beta\in V(T)$,
there is a unique   minimal path (nonrepeating sequence of  edges) connecting them. Thus it makes sense to call the number of edges in the sequence the {\em distance} between the
vertices.

Fix a tree $T$ with vertex set $V(T)$ and edge set $E(T)$. 

\begin{defn}
For each vertex $\alpha\in V(T)$, 
introduce the Euclidean space $\sL_T(\alpha) = \R^{V(T) \setminus \{\alpha\}}$
of  tuples of real numbers
$$
\xymatrix{
\{x_\gamma(\alpha)\}, 
\text{ with } \gamma\in V(T) \setminus \{\alpha\}.
}
$$
\end{defn}

\begin{defn}
For an edge $\{\alpha, \beta\}  \in E(T)$, define the  {\em $\{\alpha, \beta\} $-edge gluing} to be the quotient of the disjoint union of   Euclidean spaces 
$$
\xymatrix{
\left(\sL_T(\alpha) \coprod \sL_T(\beta)\right)/\sim
}$$ 
where we identify points
$\{x_\gamma(\alpha)\} \sim \{x_\gamma(\beta)\}$ whenever the following holds
$$
 \xymatrix{
 x_\beta(\alpha) =  x_\alpha(\beta) \geq 0
 }
 $$
 $$
 \xymatrix{
 x_\gamma(\alpha) = x_\gamma(\beta), \text{ for all } \gamma \not = \alpha, \beta \in V(T).
}
$$
\end{defn}


\begin{defn}\label{defn arb sing}

The {\em arboreal singularity} $\sL_T$ associated to a tree $T$ is the quotient  of the disjoint union
of Euclidean spaces
$$
\xymatrix{
\sL_T = \left(\coprod_{\alpha\in V(T)}  \sL_T(\alpha)\right)/\sim
}$$
by the equivalence relation  generated by the  edge gluings for all edges $\{\alpha, \beta\} \in E(T)$.

\end{defn}

\begin{example}
For the  tree $T$ with a single vertex, $\sL_T$ is a single point.
For the tree $T$ with two vertices (necessarily) connected by an edge, $\sL_T$ is the cone over three points.
For a general $A_n$-tree, see Section~\ref{sect an ex} below.
\end{example}

\begin{remark}\label{rem arb strs}
Arboreal singularities inherit two natural structures from their Euclidean space constituents.

(1) The Euclidean metric on each $\sL_T(\alpha)\subset \sL_T$ is respected by the edge gluings and hence induces
a metric on $\sL_T$ whose restriction to each $\sL_T(\alpha)\subset \sL_T$ is the original Euclidean metric.

(2) The positive dilation on each $\sL_T(\alpha)\subset \sL_T$ that sends $\{x_\gamma(\alpha)\}\mapsto \{rx_\gamma(\alpha)\}$,
for $r\in \R_{>0}$, is also 
respected by the edge gluings and hence induces
a  positive dilation on $\sL_T$ whose restriction to each $\sL_T(\alpha)\subset \sL_T$ is the original positive dilation.

The two structures satisfy the following evident compatibility. 

On the one hand, there is a unique fixed point of positive dilation denoted by $0\in \sL_T$
which we
will call the {\em central point} of $\sL_T$.
It is contained in $\sL_T(\alpha)\subset \sL_T$,
for all $\alpha\in V(T)$, with  coordinates satisfying $x_\gamma(\alpha) = 0$, for all $\gamma\in V(T)\setminus \{\alpha\}$.

On the other hand, by the {\em arboreal link} $\sL_T^\link \subset \sL_T$, we will mean the compact subspace of points unit distance from $0\in \sL_T$.

Positive dilation provides  a canonical homeomorphism
$$
\xymatrix{
 \sL_T^\link \times \R_{>0} \ar[r]^-\sim & \sL_T\setminus\{0\}
 }
$$
realizing the arboreal singularity as the cone over the link
$$
\sL_T \simeq \Cone(\sL_T^\link)
$$
where
for any space $X$, the cone is the quotient
$
\Cone(X) = X\times [0, 1) \cup_{X \times \{0\}} \pt.
$ 
\end{remark}

Next we will record two useful lemmas regarding arboreal singularities.

For any $\alpha, \beta\in V(T)$, 
there is a unique minimal path in $T$ connecting them. Suppose the path consists of $k$ edges
with  successive adjacent vertices 
$$
\gamma_0 = \alpha, \gamma_1, \ldots, \gamma_{k-1}, \gamma_k = \beta\in V(T)
$$
When $\alpha, \beta\in V(T)$ are adjacent, so that $k=1$ and there are no intermediate vertices, the following lemma reduces to the $\{\alpha, \beta\} $-edge gluing. 

\begin{lemma}\label{lem intersections}
The 
Euclidean spaces $\sL_T(\alpha)$ and $\sL_T(\beta)$ are glued inside of $\sL_T$  along the closed quadrants
where we identify points
$\{x_\gamma(\alpha)\} \sim \{x_\gamma(\beta)\}$ whenever
$$
x_{\gamma_1}(\alpha) =  x_{\alpha}(\beta), 
  x_{\gamma_{2}}(\alpha) =  x_{\gamma_{1}}(\beta), 
  \ldots,
  x_{\gamma_{k-1}}(\alpha) =  x_{\gamma_{k-2}}(\beta), 
  x_\beta(\alpha) =  x_{\gamma_{k-1}}(\beta) \geq 0
  $$
  $$
  x_\gamma(\alpha) = x_\gamma(\beta), \text{ for all } \gamma \not =\alpha,  \gamma_1, \ldots, \gamma_{k-1}, \beta\in V(T).
 $$
\end{lemma}

\begin{proof}
Note that since $T$ is acyclic, the  gluings for other edges play no role.

 Let us proceed by induction on $k$. 

For $k=1$, this is simply the $\{\alpha, \beta\} $-edge gluing.

Suppose the assertion is established for $k-1$ so that 
$\sL_T(\alpha)$ and $\sL_T({\gamma_{k-1}})$ are glued inside of $\sL_T$  along the closed quadrants
where we identify points
$\{x_\gamma(\alpha)\} \sim \{x_\gamma(\gamma_{k-1})\}$ whenever
$$
x_{\gamma_1}(\alpha) =  x_{\alpha}(\gamma_{k-1}), 
  x_{\gamma_{2}}(\alpha) =  x_{\gamma_{1}}(\gamma_{k-1}), 
  \ldots,
  x_{\gamma_{k-1}}(\alpha) =  x_{\gamma_{k-2}}(\gamma_{k-1}) \geq 0
  $$
  $$
  x_\gamma(\alpha) = x_\gamma(\gamma_{k-1}), \text{ for all } \gamma \not =\alpha,  \gamma_1, \ldots, \gamma_{k-1}
  \in V(T).
 $$

Then it suffices to observe that the $\{\gamma_{k-1}, \beta\}$-edge gluing  prescribes that 
$\sL_T({\gamma_{k-1}})$ and $\sL_T(\beta)$ are glued inside of $\sL_T$ 
where we identify points
$\{x_\gamma(\gamma_{k-1})\} \sim \{x_\gamma(\beta)\}$ whenever
$$
x_{\beta}(\gamma_{k-1}) =  x_{\gamma_{k-1}}(\beta) \geq 0
  $$
  $$
  x_\gamma(\gamma_{k-1}) = x_\gamma(\beta), \text{ for all } \gamma \not = \gamma_{k-1}, \beta\in V(T).
 $$
Composing equations, we immediately obtain the asserted equations.
\end{proof}

By a {\em terminal vertex} of a tree $T$, we will mean a vertex contained in a unique edge. By an {\em internal vertex}, we will mean a vertex that is not a terminal vertex. (By this convention,  if  $T$ consists of a single vertex alone, then the vertex is an internal vertex.)

Suppose  $T$ is a tree with $\tau\in V(T)$ a terminal vertex and $\{\tau, \alpha\}\in E(T)$  the unique edge containing $\tau$. Introduce the tree $T_\tau$ where we delete the vertex $\tau$ and the edge $\{\tau, \alpha\}$.

\begin{lemma}\label{lem  term}
There is a canonical homeomorphism
$$
\sL_T \simeq \left(\sL_{T_\tau} \times \R^{\{\tau\}}\right) \coprod_{\sL_{T_\tau}(\alpha) \times \{0\}} \left (\sL_{T_\tau}(\alpha) \times \R^{\{\alpha\}}_{\leq 0 } \right)
$$
\end{lemma}

\begin{proof}
The edge gluing for the edge $\{\tau,\alpha\} \in E(T)$ attaches the Euclidean space 
$$
\sL_T(\tau) \simeq \sL_{T_\tau}(\alpha) \times \R^{\{\alpha\}}
$$ to the product  $\sL_{T_\tau} \times \R^{\{\tau\}}$ along the closed subspace
$$
\sL_{T_\tau}(\alpha)  \times \R^{\{\alpha\}}_{\geq 0 } 
\subset  \sL_{T_\tau}(\alpha) \times \R^{\{\alpha\}}
$$ 
The  gluing of the lemma results from removing the redundant open subspace  
$$
\sL_{T_\tau}(\alpha)  \times \R^{\{\alpha\}}_{> 0 }
\subset \sL_{T_\tau}(\alpha) \times \R^{\{\alpha\}}
$$  
and only attaching the closed complement  
$$
\sL_{T_\tau}(\alpha)  \times \R^{\{\alpha\}}_{\leq 0 }
\subset 
 \sL_{T_\tau}(\alpha) \times \R^{\{\alpha\}}
$$
\end{proof}

\begin{remark}
The choice of a terminal vertex is not canonical, but the collection of all terminal vertices is. For a more invariant statement, one could simultaneously apply the above lemma to all terminal vertices (as long as there are  three or more vertices).
\end{remark}
%
%
%
%
%
%
%

\begin{corollary}\label{cor bouquet}
The arboreal link $\sL_T^\link$ is homotopy equivalent to a bouquet of $|V(T)|$ spheres each of dimension $|V(T)|-2$.
\end{corollary}

\begin{proof}
Let us adopt the setting and notation of the previous lemma.
 
 By induction,  $\sL_{T_\tau}^\link$ 
 is homotopy equivalent to a bouquet of $|V(T)|-1$ spheres each of dimension $|V(T)|-3$,
 and hence the suspension
 $$
\xymatrix{
\Sigma (\sL_{T_\tau}^\link) \simeq  (\sL_{T_\tau} \times\R^{\{\tau\}})^\link
}
$$ 
 is homotopy equivalent to a bouquet of $|V(T)|-1$ spheres each of dimension $|V(T)|-2$.

By the previous lemma,  $\sL^\link_T$ results from starting with $\Sigma (\sL_{T_\tau}^\link)$   
and attaching the $(|V(T)|-2)$-cell $L_{T_\tau}(\alpha)$ along the inclusion of its boundary  
$(|V(T)|-3)$-sphere 
$$
 \sL_{T_\tau}(\alpha)^\link 
 \subset \Sigma (\sL_{T_\tau}^\link) 
$$
induced by the inclusion $ \sL_{T_\tau}(\alpha) 
 \subset \sL_{T_\tau}$.
 Therefore $\sL_T^\link$ is homotopy equivalent to a bouquet of $|V(T)|$ spheres each of dimension $|V(T)|-2$.
 \end{proof}


\subsection{Combinatorial description }
Let us first review some terminology.

Given a graph $G$, by a {\em subgraph} $S\subset G$, we will mean a full subgraph (or vertex-induced subgraph) in the sense that its vertices are a subset $V(S) \subset V(G)$ and
its edges are the subset $E(S) \subset E(G)$ such that $\{\alpha,\beta\} \in E(S)$ if and only if $\{\alpha,\beta\}\in E(G)$ and $\alpha, \beta\in V(S)$.
By the {\em complementary subgraph}
$G\setminus S \subset G$, we will mean the full subgraph on the complementary vertices $V(T\setminus S) = V(T) \setminus V(S)$.

Given a tree $T$, any subgraph $S\subset T$ is a disjoint union of trees. By a {\em subtree} $S \subset T$, we will mean a subgraph that is a tree.
The complementary subgraph
$T\setminus S \subset T$ is not necessarily a tree but in general a disjoint union of subtrees.
Given a subtree $S\subset T$, and a vertex $\alpha\in V(T \setminus S)$, there is a unique  vertex $\gamma\in V(S)$
{\em nearest} to $\alpha$.

Given a tree $T$, by a {\em quotient tree} $T\twoheadrightarrow Q$, we will mean a tree $Q$ with a surjection $V(T)\twoheadrightarrow V(Q)$ such that each fiber comprises the vertices of a subtree of $T$. 
We will refer to such subtrees as the {\em fibers} of the quotient $T\twoheadrightarrow Q$.
Given a vertex $\alpha\in V(T)$,
we will sometimes write $\ol\alpha\in V(Q)$ for its image, and $T_{\ol\alpha}\subset T$ for the fiber containing $\alpha$.

By a {\em partition} of a tree $T$, we will mean a  collection of subtrees $T_i\in T$, for $i\in I$, that are disjoint $V(T_i) \cap V(T_j) = \emptyset$, for $i\not = j$, and cover $V(T) = \coprod_{i\in I} V(T_i)$.
Note that  the data of a quotient $T\twoheadrightarrow Q$ is equivalent to the partition of $T$ into  the fibers.

%

%

\medskip

Now let $T$ be a tree with arboreal singularity $\sL_T$. A  point $x\in \sL_T$ defines the following invariants.

First, we introduce the function 
$$
\xymatrix{
v_x:V(T) \ar[r] & \{\mathit{yes}, \mathit{no}\}\\
}
$$
such that $v_x (\alpha)= \mathit{yes}$ when $x \in  \sL_T(\alpha) \subset \sL_T$, and $v_x(\alpha) = \mathit{no}$ 
 when $x \not \in \sL_T(\alpha)  \subset \sL_T$.

%
%
%

Define the subgraph $S\subset T $ to consist of those vertices $\alpha\in V(T)$ such that $v_x(\alpha) = \yes$, and those edges $\{\alpha,\beta\}\in E(T)$
such that $v_x(\alpha) = v_x(\beta) = \mathit{yes}$.

\begin{lemma}
$S$ is a tree.
\end{lemma}

\begin{proof}
We must show $S$ is connected (it is a subgraph of $T$ so clearly acyclic). Suppose $\alpha, \beta\in V(S)$ so that $x\in
  \sL_T(\alpha) \cap \sL_T(\beta)  \subset \sL_T $. 
  Suppose the unique minimal path in $T$ connecting them consists of $k$ edges
with  successive vertices $\gamma_0 = \alpha, \gamma_1, \ldots, \gamma_{k-1}, \gamma_k = \beta\in V(T)$. 
 By Lemma~\ref{lem intersections}, for all $i= 1, \ldots, k-1$, we see that $x$ is contained
  in the intermediate Euclidean spaces $ \sL_T({\gamma_i})  \subset \sL_T $,  and thus $\alpha, \beta\in V(S)$ are connected by a path in $S$.
\end{proof}

\begin{remark}
Consider the complementary graph $T\setminus S$  
consisting of those vertices $\alpha\in V(T)$ such that $v_x(\alpha) = \no$, and those edges $\{\alpha,\beta\}\in E(T)$
such that $v_x(\alpha) = v_x(\beta) = \no$.
In general, it   is the disjoint union of subtrees   $N_i \subset T$, for $i\in I$, but not necessarily connected. 
\end{remark}

Let us continue with the invariants of a  point $x\in \sL_T$.

Observe that $\{\alpha,\beta\} \in E(S)$ means $x\in  \sL_T(\alpha)\cap \sL_T(\beta) \subset \sL_T$, and the $\{\alpha,\beta\}$-edge gluing implies an equality of non-negative coordinates
$x_\beta(\alpha) = x_\alpha(\beta) \geq 0$ evaluated at $x$.

Next, we introduce the function 
$$
\xymatrix{
e_x:E(S) \ar[r] & \{0, +\}
}
$$
such that $e_x (\{\alpha,\beta\})= 0$ when $x_\beta(\alpha) = x_\alpha(\beta) = 0$  evaluated at $x$,
and 
 $e_x(\{\alpha,\beta\})= +$ when  $x_\beta(\alpha) = x_\alpha(\beta)> 0$   evaluated at $x$.
%

%
%
%
%
%

Define the tree $R$ to be the quotient of $S$ where we contract those edges
  $\{\alpha,\beta\} \in E(S)$ such that $e(\{\alpha,\beta\}) = +$. 
  Thus the vertex set
    $V(R)$ is the  quotient of $V(S)$
where we identify vertices
 $\alpha, \beta\in V(S)$ that can be connected by a path through edges
$\{\alpha,\beta\} \in E(S)$ such that $e_x(\{\alpha,\beta\}) = + $. 
The edge set $E(R)$ can be taken to 
consist of those edges  $\{\alpha,\beta\} \in E(S)$ such that $e(\{\alpha,\beta\}) = 0 $.

Altogether, we see that  the  point $x\in L_T$ defines a correspondence of trees
$$
\xymatrix{
R &  \ar@{->>}[l]_-q  S \ar@{^(->}[r]^-p & T
}
$$
where $p$ is the inclusion of  a subtree, and $q$ is a quotient map of trees.
Here and in what follows, we take such correspondences up to the strictest  notion of equivalence:
two such correspondences 
$$
\xymatrix{
R &  \ar@{->>}[l]_-q  S \ar@{^(->}[r]^-p & T
&
R &  \ar@{->>}[l]_-{q'}  S' \ar@{^(->}[r]^-{p'} & T
}
$$
are equivalent if and only if there is a bijection $b:S'\risom S$ such that $q' = q\circ b$, $p' = p \circ b$.

Here is a useful reformulation of the data of such a correspondence.

\begin{lemma}\label{lem:corrpart}
The set of  correspondences
$$
\xymatrix{
R &  \ar@{->>}[l]_-q  S \ar@{^(->}[r]^-p & T
}
$$
where $p$ is the inclusion of  a subtree, and $q$ is a quotient map of trees,
 is in natural bijection with the set of partitions of $T$ into a collection of subtrees 
$$
\left(\{N_i\}_{i\in I}, \{F_{j}\}_{j\in J}\right)
$$ 
such that $J$ is nonempty, 
and for any $i \in I$, the complement  $T\setminus  N_i$ is connected.
\end{lemma}

\begin{proof}
Given such a correspondence, define  the  
 subtrees $N_i\subset T$, for $i\in I$,  to be the connected components of the complementary graph $T\setminus S$. 
Define  the  
 subtrees $F_j\subset T$, for $j\in J$, to be the fibers of the quotient map $q:S\twoheadrightarrow R$.
 Since $S$ is nonempty, $J$ is nonempty.
 
 Suppose there is some $i \in I$ such that    $T\setminus  N_i$ is disconnected. 
  Since $S$ is connected,  it lies in one of the components of $T\setminus N_i$.
   But then there is another component of $T\setminus  N_i$ contained in $T\setminus S$ and
   connected to $N_i$, contradicting
 that $N_i$ itself is a component of $T\setminus S$.
 
 Conversely, given such a partition, the full subgraph $S = \coprod_{j\in J} F_j \subset T$ is nonempty since $J$
 is nonempty. Furthermore, $S$ is connected, and hence a subtree, else there is some $i\in I$ such that $T\setminus N_i$ is disconnected. Finally, take the quotient map $q:S\twoheadrightarrow R$  to be that with fibers given by  
 $F_{j} \subset S$, for $j\in J$.
 \end{proof}

We will show that the arboreal singularity $\sL_T$ is the cone over a regular cell complex with each cell the subspace of points leading to a given correspondence.
We will arrive at this in Theorem~\ref{thm reg cc} below, but first observe that such subspaces and their closures are naturally convex polyhedra.

\begin{defn}
Let $T$ be a tree with associated arboreal singularity $\sL_T$.

Define $\sL_T(\fp)\subset \sL_T$ to be the subspace of points leading to a given correspondence
$$
\xymatrix{
\fp = (R &  \ar@{->>}[l]_-q  S \ar@{^(->}[r]^-i & T)
}
$$
where $i$ is the inclusion of  a subtree, and $q$ is a quotient map of trees.

Define the rank $\rho(\fp) = |V(T)| - |V(R)|$.
\end{defn}

\begin{prop}\label{prop strat}
Let $T$ be a tree with associated arboreal singularity $\sL_T$.

The subspace $\sL_T(\fp) \subset \sL_T$ 
of points leading to a given correspondence
$$
\xymatrix{
\fp = (R &  \ar@{->>}[l]_-q  S \ar@{^(->}[r]^-i & T)
}
$$
is an open cell of dimension $\rho(\fp) = |V(T)| - |V(R)|$. Its closure is naturally  a convex polyhedron in a Euclidean space, and in fact
cut out by explicit equalities and inequalities on coordinate functions (appearing in the proof below).

\end{prop}

\begin{proof}
Fix any $\alpha\in V(S)$ and observe that $\sL_T(\fp)\subset \sL_T(\alpha)$ since  $v_x(\alpha) = \yes$.

Recall that points $x\in  \sL_T(\alpha)$ consist of tuples of real numbers
$$
\xymatrix{
\{x_\gamma(\alpha)\}, 
\text{ with } \gamma\in V(T) \setminus \{\alpha\}.
}
$$



We claim that those points $x\in  \sL_T(\alpha)$ that lie in $\sL_T(\fp)\subset \sL_T(\alpha)$ are precisely cut out by the 
following equations
on their coordinates $x_\gamma(\alpha)$ depending on the location of $\gamma$.
 
\begin{enumerate}

\item Suppose $\gamma$ lies in the fiber $F_{\ol\alpha}\subset S$ containing $\alpha$. Then we have $x_\gamma(\alpha)>0$.

 \item Suppose $\gamma$ lies in $S$ but not  in the fiber $F_{\ol\alpha}\subset S$ containing $\alpha$. Suppose also that $\gamma$
 is the nearest vertex to $\alpha$ within the fiber $F_{\ol\gamma}\subset S$ containing $\gamma$. Then we have $x_\gamma(\alpha)=0$.

\item Suppose $\gamma$  lies in $S$ but not  in the fiber $F_{\ol\alpha}\subset S$ containing $\alpha$. Suppose also that $\gamma$
 is not the nearest vertex to $\alpha$ in the fiber $F_{\ol\gamma} \subset S$ containing $\gamma$. Then we have $x_\gamma(\alpha)>0$.

\item Suppose $\gamma$ lies in $T\setminus S$ so that $\gamma$ is in some  subtree $N_i\subset T\setminus S$. Suppose also that $\gamma$  is the nearest vertex to $\alpha$ within $N_i$. Then we have $x_\gamma(\alpha)<0$.

\item Suppose $\gamma$ lies in   $T\setminus S$ so that $\gamma$ is in some  subtree $N_i\subset T\setminus S$. Suppose also that $\gamma$  is not the nearest vertex to $\alpha$ within $N_i$. Then we allow $x_\gamma(\alpha)$ to be arbitrary.
\end{enumerate}

To confirm this, on the one hand, by Lemma~\ref{lem intersections}, if $x_{\gamma}(\alpha) <0$, then $x\not\in \sL_T(\gamma)$, 
and so $\gamma$ does not lie in $S$.
 On the other hand, suppose $\gamma$ lies in $T\setminus S$.
Consider the minimal path connecting $\alpha$ and $\gamma$,
and let $\gamma'$ be the closest point to $\alpha$ that lies on the path and in $T\setminus S$.
By Lemma~\ref{lem intersections}, we have $x_{\gamma'}(\alpha)<0$, and 
 $x_{\gamma}(\alpha)$ can be arbitrary if $\gamma \not = \gamma'$.
Thus $x\in \sL_T(\alpha)$ leads to   
the right half of the correspondence
$$
\xymatrix{
S \ar@{^(->}[r] & T
}
$$
if and only if the following coarser equations hold.

\begin{enumerate}

\item[($1'$)] Suppose $\gamma$ lies in the fiber $F_{\ol\alpha}\subset S$ containing $\alpha$. Then we have $x_\gamma(\alpha)\geq 0$.

 \item[($2'$)] Suppose $\gamma$ lies in $S$ but not  in the fiber $F_{\ol\alpha}\subset S$ containing $\alpha$. Suppose also that $\gamma$
 is the nearest vertex to $\alpha$ within the fiber $F_{\ol\gamma}\subset S$ containing $\gamma$. Then we have $x_\gamma(\alpha)\geq 0$.

\item[($3'$)] Suppose $\gamma$  lies in $S$ but not  in the fiber $F_{\ol\alpha}\subset S$ containing $\alpha$. Suppose also that $\gamma$
 is not the nearest vertex to $\alpha$ in the fiber $F_{\ol\gamma} \subset S$ containing $\gamma$. Then we have $x_\gamma(\alpha)\geq 0$.

\item[($4$)] Suppose $\gamma$ lies in $T\setminus S$ so that $\gamma$ is in some  subtree $N_i\subset T\setminus S$. Suppose also that $\gamma$  is the nearest vertex to $\alpha$ within $N_i$. Then we have $x_\gamma(\alpha)<0$.

\item[($5$)] Suppose $\gamma$ lies in   $T\setminus S$ so that $\gamma$ is in some  subtree $N_i\subset T\setminus S$. Suppose also that $\gamma$  is not the nearest vertex to $\alpha$ within $N_i$. Then we allow $x_\gamma(\alpha)$ to be arbitrary.

\end{enumerate}

Now it remains to determine that in cases ($1'$) and ($3'$) the coordinate $x_\gamma(\alpha)$ should be positive as in cases ($1$) and ($3$), and in case ($2'$) it should be zero as in case ($2$). 

Again apply Lemma~\ref{lem intersections} to  see that for $x\in \sL_T(\alpha)$ satifying the above coarser equations, and for  $\gamma\in V(S)$ so that 
$x\in \sL_T(\gamma)$, we have the equality of functions $x_\gamma(\alpha) = x_{\tilde \gamma}(\gamma)$, where $\tilde \gamma \in V(S)$ is the vertex adjacent to $\gamma$
and one edge nearer to $\alpha$. Thus by definition $x\in \sL_T(\alpha)$ also leads to 
the left half of the correspondence
$$
\xymatrix{
R &  \ar@{->>}[l]  S 
}
$$
if and only if 
in cases ($1'$) and ($3'$) the coordinate $x_\gamma(\alpha)$ is positive, and
in case ($2'$) it is zero.

Finally, for the dimension assertion, recall that $\dim \sL_T(\alpha) =|V(T)| -1$, and note that there are precisely $|V(R)|-1$ equalities in the above equations given by case (2).
 \end{proof}

Next we will  calculate the closure relations among the cells.
Consider a pair of correspondences of trees
$$
\xymatrix{
\fq=(P &  \ar@{->>}[l]_-c  Q \ar@{^(->}[r]^-i & R)
&
\fp=(R &  \ar@{->>}[l]_-d  S \ar@{^(->}[r]^-j & T)
}
$$
where $i, j$ are inclusions of  subtrees, and $c, d$ are quotient maps of trees.
Define the composed correspondence
$$
\xymatrix{
\fq\circ \fp = (P &  \ar@{->>}[l]_-{\tilde c}  Q \times_R S \ar@{^(->}[r]^-{\tilde \jmath} & T)
}
$$
by  forming the fiber product
$$
\xymatrix{
 & &\ar@{->>}[dl]_-{\pi_Q}  Q \times_R S \ar@{^(->}[dr]^-{\pi_S} & & \\
& \ar@{->>}[dl]_-c Q  \ar@{^(->}[dr]^-{i}  & &  \ar@{->>}[dl]_-d  S \ar@{^(->}[dr]^-{j}  & \\
P &  & R &   & T
}
$$
and setting $\tilde c = c\circ \pi_Q$, $\tilde \jmath = j \circ \pi_S$, 
where $\pi_Q, \pi_S$ are the natural projections.
%
%

\begin{defn}
Let $T$ be a tree.

1) Let $\fP_T$ denote the poset whose elements are
 correspondences of trees
$$
\xymatrix{
\fp=(R &  \ar@{->>}[l]_-d  S \ar@{^(->}[r]^-j & T)
}
$$
where $j$ is the inclusion of a  subtree, and $d$ is a quotient map of trees.

We define the order relation on correspondences 
$$
\xymatrix{
\fp=(R &  \ar@{->>}[l]_-d  S \ar@{^(->}[r]^-j & T) \geq  \fp'=(R' &  \ar@{->>}[l]_-{d'}  S' \ar@{^(->}[r]^-{j'} & T)
}
$$
and say the second correspondence refines the first,
if  there is a 
third correspondences of trees
$$
\xymatrix{
\fq=(R &  \ar@{->>}[l]_-c  Q \ar@{^(->}[r]^-i & R')
}
$$
where $i$ is the inclusion of a  subtree, and $c$ is a quotient map of trees,
such that 
$$
\fp = \fq\circ \fp'.
$$

2) Let $\mathfrak Q_T$ denote the poset whose elements are partitions of $T$ into
a   collection of subtrees 
$$
\left(\{N_i\}_{i\in I}, \{F_{j}\}_{j\in J}\right)
$$ 
such that $J$ is nonempty, 
and for any $i \in I$, the complement  $T\setminus  N_i$ is connected.

We define the order relation on partitions
$$
\xymatrix{
(\{N_i\}_{i\in I}, \{F_{j}\}_{j\in J}) \geq (\{N'_k\}_{k\in K}, \{F'_{\ell}\}_{\ell\in L})
}
$$ 
and say the second partition refines the first,
if each $N'_k$ lies in some $N_i$ and each $F'_{\ell}$ lies in some $N_i$ or some $F_j$.
\end{defn}

\begin{remark}
There is the unique minimum  $\fp_0\in \fP_T$ given by the identity correspondence
$$
\xymatrix{
\fp_0=(T &  \ar@{->>}[l]_-=  T \ar@{^(->}[r]^-= & T)
}
$$
It indexes the  cell $\sL_{\fp_0} \subset \sL_T$ comprising the central point $0\in \sL_T$ alone.
\end{remark}

We have the following upgrading of Lemma~\ref{lem:corrpart}.

\begin{lemma}\label{lem:partition order}
The bijection of Lemma~\ref{lem:corrpart} is a poset isomorphism $\fP_T \simeq \mathfrak Q_T$.
\end{lemma}

\begin{proof} Suppose  
$$
\xymatrix{
\fp=(R &  \ar@{->>}[l]_-d  S \ar@{^(->}[r]^-j & T) &
  \fp'=(R' &  \ar@{->>}[l]_-{d'}  S' \ar@{^(->}[r]^-{j'} & T)
}
$$
due to $\fp = \fq\circ \fp'$ with
$$
\xymatrix{
\fq=(R &  \ar@{->>}[l]_-c  Q \ar@{^(->}[r]^-i & R')
}
$$
Then the corresponding partitions satisfy
$$
\xymatrix{
(\{N_i\}_{i\in I}, \{F_{j}\}_{j\in J}) \geq (\{N'_k\}_{k\in K}, \{F'_{\ell}\}_{\ell\in L})
}
$$
since $j$ factors through $j'$, so each $N'_k$ lies in some $N_i$, and $d$ factors through the base change of $d'$
by the inclusion $i$,
so each $F_\ell'$ lies in some $F_j$.

Conversely, suppose
$$
\xymatrix{
(\{N_i\}_{i\in I}, \{F_{j}\}_{j\in J}) \geq (\{N'_k\}_{k\in K}, \{F'_{\ell}\}_{\ell\in L})
}
$$
and consider the corresponding correspondences 
$$
\xymatrix{
\fp=(R &  \ar@{->>}[l]_-d  S \ar@{^(->}[r]^-j & T) &
  \fp'=(R' &  \ar@{->>}[l]_-{d'}  S' \ar@{^(->}[r]^-{j'} & T)
}
$$
Since each $N'_k$ lies in some $N_i$, and $S' = T \setminus \coprod_{k\in K} N'_k$, $S = T \setminus \coprod_{i\in I} N_i$, we have $S\subset S'$, or in other words,
$j$ factors through $j'$. Define the quotient map $S\twoheadrightarrow Q$ to be that with fibers given by $F'_\ell 
\cap S$, for $\ell\in L$. Then the inclusion $S\hookrightarrow S'$ descends to an inclusion $i:Q\hookrightarrow R'$.
Finally, the quotient map $d:S\twoheadrightarrow R$ factors through $S\twoheadrightarrow Q$,
providing a quotient map $c:Q\twoheadrightarrow R$, since
each $F_\ell'$ lies in some $F_j$. Thus $\fp\geq \fp'$ due to $\fp = \fq\circ \fp'$ with
$$
\xymatrix{
\fq=(R &  \ar@{->>}[l]_-c  Q \ar@{^(->}[r]^-i & R')
}
$$
  \end{proof}

\begin{prop}
Two elements $\fp, \fp'\in \fP_T$ indexing cells $\sL_T({\fp}), \sL_T({\fp'}) \subset \sL_T$ satisfy 
$\fp\geq \fp'$ if and only if $\sL_T({\fp'})$ intersects the closure of $\sL_T(\fp)$. If this holds, then in fact 
$\sL_T({\fp'})$ lies in the closure of $\sL_T(\fp)$.
\end{prop}

\begin{proof}
We will proceed in the language of partitions as in Lemmas~\ref{lem:corrpart} and \ref{lem:partition order} though
one could equally well translate the arguments back into the language of correspondences.

Consider two element $\fp,  \fp'\in \fP_T$ representing respective partitions 
$$
\xymatrix{
(\{N_i\}_{i\in I}, \{F_{j}\}_{j\in J})
&
(\{N'_k\}_{k\in K}, \{F'_{\ell}\}_{\ell\in L})
}$$
comprising subtrees of complementary components and fibers.

Suppose $\fp\geq \fp'$ so that  the partition associated to $\fp'$   refines that associated to $\fp$. We will show that
$\sL_T({\fp'})$ is in the closure of $\sL_T(\fp)$. 

 Choose any fiber $F_j$.  Then $F_j$ contains some fiber $F'_\ell$.
Fix any vertex $\alpha \in V(F'_\ell)$. Then we have $\alpha\in V(F_j)$ as well. Thus we have
$$
\sL_T(\fp), \sL_T(\fp') \subset \sL_T(\alpha)
$$
Returning to the proof of Proposition~\ref{prop strat}, we find explicit equations for these subspaces to contain
 a point $x\in \sL_T(\alpha)$ in terms of its coordinates 
 $$
\xymatrix{
\{x_\gamma(\alpha)\}, 
\text{ with } \gamma\in V(T) \setminus \{\alpha\}.
}
$$
We have the following possibilities  depending on the location of $\gamma$. 
The fact that the partition associated to $\fp'$   refines that associated to $\fp$
implies simple constraints on the location of $\gamma$ with respect to the partitions and $\alpha$.

\begin{enumerate}

\item Suppose $\gamma$ lies in the fiber $F_{\ol\alpha}\subset S$ containing $\alpha$ so that 
$$x\in \sL_T(\fp) \implies x_\gamma(\alpha)>0
$$
Then $\gamma$ must lie in some fiber $F'_{\ol\gamma}\subset F_{\ol\alpha}$ 
so that 
$$
x\in \sL_T(\fp') \implies x_\gamma(\alpha)\geq 0
$$

 \item Suppose $\gamma$ lies in $S$ but not  in the fiber $F_{\ol\alpha}\subset S$ containing $\alpha$. Suppose also that $\gamma$
 is the nearest vertex to $\alpha$ within the fiber $F_{\ol\gamma}\subset S$ containing $\gamma$ so that
  $$x\in \sL_T(\fp) \implies x_\gamma(\alpha)=0
  $$
Then $\gamma$ must lie in some fiber $F'_{\ol\gamma}\subset F_{\ol\gamma}$ 
and 
$\gamma$
 must be the nearest vertex to $\alpha$ within the fiber $F'_{\ol\gamma}\subset  F_{\ol\gamma}$
 so that
 $$x\in \sL_T(\fp') \implies x_\gamma(\alpha)= 0
 $$

\item Suppose $\gamma$  lies in $S$ but not  in the fiber $F_{\ol\alpha}\subset S$ containing $\alpha$. Suppose also that $\gamma$
 is not the nearest vertex to $\alpha$ in the fiber $F_{\ol\gamma} \subset S$ containing $\gamma$ so that 
 $$
 x\in \sL_T(\fp) \implies x_\gamma(\alpha)>0
 $$
 Then $\gamma$ must lie in some fiber $F'_{\ol\gamma}\subset F_{\ol\gamma}$ 
so that 
$$
x\in \sL_T(\fp') \implies x_\gamma(\alpha)\geq 0
$$

\item Suppose $\gamma$ lies in $T\setminus S$ so that $\gamma$ is in some  subtree $N_i\subset T\setminus S$. Suppose also that $\gamma$  is the nearest vertex to $\alpha$ within $N_i$ so that 
$$
x\in \sL_T(\fp) \implies x_\gamma(\alpha)<0
$$
Then either (a) $\gamma$ must  lie in some fiber $F'_{\ol\gamma}\subset N_i$ 
and 
$\gamma$
 must be the nearest vertex to $\alpha$ within the fiber $F'_{\ol\gamma}\subset  N_i$
 so that
 $$x\in \sL_T(\fp') \implies x_\gamma(\alpha)= 0
 $$
or (b) $\gamma$ must  lie in some some subtree $N'_k  \subset N_i$ 
and 
$\gamma$
 must be the nearest vertex to $\alpha$ within the subtree $N'_k  \subset N_i$
 so that
 $$x\in \sL_T(\fp') \implies x_\gamma(\alpha)< 0
 $$

\item Suppose $\gamma$ lies in   $T\setminus S$ so that $\gamma$ is in some  subtree $N_i\subset T\setminus S$. Suppose also that $\gamma$  is not the nearest vertex to $\alpha$ within $N_i$ so that  
$$x\in \sL_T(\fp) \implies x_\gamma(\alpha) \text{ arbitrary}
 $$
Thus we need not worry about the possible constraints imposed  by $ x\in \sL_T(\fp')$.

\end{enumerate}

From the above equations, we conclude that
$\sL_T({\fp'})$ is in the closure of $\sL_T(\fp)$.

Conversely, suppose $\sL_T({\fp'})$ intersects the closure of $\sL_T(\fp)$ at some point $x\in \sL_T$.  We will show that
 the partition of $\fp'$   refines that of $\fp$.
  To begin, choose any $\alpha\in V(T)$ such that
   $\sL_T(\fp) \subset \sL_T(\alpha)$. Hence  the closure of $\sL_T(\fp)$ also lies in $\sL_T(\alpha)$. Thus  
$\sL_T(\fp')$ intersects $\sL_T(\alpha)$, and hence $\sL_T(\fp')\subset \sL_T(\alpha)$.
Thus we have $x\in \sL_T(\alpha)$ with coordinates 
 $$
\xymatrix{
\{x_\gamma(\alpha)\}, 
\text{ with } \gamma\in V(T) \setminus \{\alpha\}.
}
$$
 
Now we will proceed by contradiction. Suppose some subtree $N'_k \subset T\setminus S'$ is not contained in any subtree $N_i \subset T\setminus S$. Then the vertex $\beta$ within $N'_k$ closest to $\alpha$  must lie in some fiber $F_j$. But turning to the equations for $x_\beta(\alpha)$ of Proposition~\ref{prop strat}, 
and
 applying case (1), (2), or (3) to $\sL_T(\fp)$ and case (4) to $\sL_T(\fp')$, we find
$$
\xymatrix{
x\in \ol{\sL_T(\fp)} \implies x_\beta(\alpha) \geq 0
&
x\in \sL_T(\fp') \implies x_\beta(\alpha)< 0
}$$
and hence a contradiction.
 
 Next suppose some fiber $F'_\ell$ is not contained in any subtree $N_i$ or fiber $F_j$. 
 Then there is a vertex $\beta$ within $F'_\ell$ that  is not  the nearest to $\alpha$ 
 within $F'_\ell$ but is the nearest to $\alpha$ within whichever subtree $N_i$ or fiber $F_j$ contains $\beta$.
 Again turning to the equations for $x_\beta(\alpha)$ of Proposition~\ref{prop strat}, 
and
 applying case (2) or  (4) to $\sL_T(\fp)$ and case (1) or  (3) to $\sL_T(\fp')$, we find
 $$
\xymatrix{
x_\beta(\alpha)\in \ol{\sL_T(\fp)} \implies  x_\beta(\alpha)\leq  0
&
x_\beta(\alpha)\in \sL_T(\fp') \implies  x_\beta(\alpha)> 0
}$$
and hence a contradiction.
 
 Finally, for the last assertion, we have shown that
 if $\fp\geq \fp'$ then $\sL_T({\fp'})$ lies in the closure of $\sL_T(\fp)$, and also that 
if  $\sL_T({\fp'})$ intersects  the closure of $\sL_T(\fp)$ then  $\fp\geq \fp'$. Thus we conclude that if
 $\sL_T({\fp'})$ intersects  the closure of $\sL_T(\fp)$ then $\sL_T({\fp'})$ lies in the closure of $\sL_T(\fp)$.
\end{proof}

Recall from Remark~\ref{rem arb strs} that the arboreal singularity $\sL_T$ inherits a natural metric from its Euclidean constituents,
and the arboreal link $\sL_T^\link \subset \sL_T$ refers to the compact subspace of points unit distance from the center
 $0\in \sL_T$.
Positive dilation provides  a canonical homeomorphism
$$
\xymatrix{
 \sL_T^\link \times \R_{>0} \ar[r]^-\sim & \sL_T\setminus\{0\}
 }
$$
realizing the arboreal singularity as the cone over the link
$$
\sL_T \simeq \Cone(\sL_T^\link)
$$

Note that all of the subspaces $\sL_T(\fp)\subset  \sL_T$ are invariant under positive dilation.

\begin{defn}

Fix an element $\fp\in\fP_T\setminus\{\fp_0\}$.

Introduce
 the open
 cell
$$
\xymatrix{
\sL_T^\link(\fp) =  \sL^\link_T \cap\sL_T(\fp) 
}
$$ 
of dimension $\rho(\fp)-1$.
\end{defn}

Now the two preceding propositions coupled with general theory~(\cite{bjorner, mccrory}) immediately imply the following.
Recall that a {\em finite regular cell complex} is a Hausdorff space $X$
with a finite collection of closed cells $c_i \subset X$ whose interiors $c_i^\circ \subset c_i$ provide a partition of $X$ and boundaries $\partial c_i \subset X$
are unions of cells. A finite regular cell complex $X$ has the {\em intersection property} if the intersection of any two cells $c_i, c_j\subset X$ is 
either another cell or empty. (This holds for example if the closed cells are convex polyhedra in Euclidean spaces glued along convex subspaces as found in Proposition~\ref{prop strat}.)
The {\em order complex} of a poset is the natural simplicial complex with simplices the finite totally-ordered chains of the poset.
(Other useful references include~ \cite{bb, wachs}.)

\begin{thm}\label{thm reg cc}
Let $T$ be a tree with associated arboreal link $\sL^\link_T$.

The decomposition of the arboreal link $\sL^\link_T$ into the cells $\sL^\link_T(\fp)$, for
$\fp\in \fP_T \setminus \{\fp_0\}$,  is a regular cell complex
with the intersection property.  It is homeomorphic to the order  complex of  $\fP_T \setminus \{\fp_0\}$.
\end{thm}

Before continuing on, let us also  record the local structure of arboreal singularities.

\begin{defn}
Fix an element  $\fp\in \fP_T$.

(1) Introduce the poset
$$
\xymatrix{
\fP_T(\geq\fp) = \{\fq\in  \fP_T \,|\, \fq\geq \fp\}
}$$
equipped with the induced partial order.

(2) Introduce the open neighborhood 
$$
\xymatrix{
\sL_T(\geq\fp) = \coprod_{\fq \in \fP_T(\geq\fp)} \sL_T(\fq)  \subset \sL_T
}$$
of the cell $\sL_T(\fp)
\subset \sL_T$.
\end{defn}

\begin{lemma}
Given an element  $\fp\in \fP_T$ representing a correspondence
$$
\xymatrix{
\fp=(R &  \ar@{->>}[l]_-d  S \ar@{^(->}[r]^-j & T)
}
$$
the natural poset map is an isomorphism
$$
\xymatrix{
\fP_R\ar[r]^-\sim &\fP_T(\geq \fp)
&
\fq \ar@{|->}[r] & \fq\circ \fp
}$$
\end{lemma}

\begin{proof}
The map is surjective by the definition of the partial order and we must show it is injective.
Suppose we have $\fq,\fq'\in \fP_R$ representing correspondences
$$
\xymatrix{
\fq=(P &  \ar@{->>}[l]_-d  Q \ar@{^(->}[r]^-j & R)
&
\fq'=(P' &  \ar@{->>}[l]_-d  Q' \ar@{^(->}[r]^-j & R)
}
$$
such that $\fq\circ\fp \simeq \fq'\circ\fp$. So we may assume that $P=P'$ and need to show that $Q, Q'\subset R$ are the same subset.
But $\fq\circ\fp \simeq \fq'\circ\fp$ implies that $Q\times_R S = Q'\times_R S \subset S$, and hence the surjection $S\twoheadrightarrow R$ implies that $Q=Q'\subset S$.
\end{proof}
%
%
%

\begin{corollary}\label{cor local}
Let $T$ be a tree with associated arboreal singularity $\sL_T$.

Fix an element  $\fp\in \fP_T$ indexing the cell $\sL_T(\fp) \subset \sL_T$ and  open neighborhood $ \sL_T(\geq \fp)\subset \sL_T$.

The poset isomorphism 
$$
\xymatrix{
\fP_R\ar[r]^-\sim &\fP_T(\geq \fp)
}
$$ 
induces a homeomorphism
$$
\xymatrix{
\sL_T(\fp) \times \sL_R \ar[r]^-\sim & \sL_T(\geq \fp)
}
$$ 
\end{corollary}


\subsection{Example: $A_n$-trees}\label{sect an ex}

By the {\em $A_n$-tree $T_n$}, we will mean the tree with $n$ vertices labelled $v_1, \ldots, v_n$ and an edge
connecting the vertices $v_i$ and $v_{i+1}$, for all $i = 1, \ldots, n-1$.

Let $\Delta^n$ denote the $n$-simplex. Let $[n]= \{0,1,  \ldots, n\}$ denote its vertices so that the subsimplices of $\Delta^n$ are in natural bijection with nonempty subsets of $[n]$.

 Let $sk_{n-2} \Delta^n$ denote the  $(n-2)$-skeleton of $\Delta^n$. The subsimplices of $sk_{n-2} \Delta^n$ 
 are in natural bijection with nonempty subsets of $[n]$ containing at most $n-1$ elements.

\begin{prop}
There is an identification of regular cell complexes
$$
\sL_{T_n} \simeq \Cone(sk_{n-2} \Delta^n)
$$
\end{prop}

\begin{proof}
Let  $\fP_n$ denote the poset of  nonempty subsets of $[n]$ containing at most $n-1$ elements
with the standard partial order: for $A, A' \subset [n]$, we set $A\geq A'$ if and only if $A\supset A'$.

By Theorem~\ref{thm reg cc},  it suffices to establish an isomorphism of posets 
$$
\xymatrix{
\varphi:\fP_{T_n} \setminus\{\fp_0\}  \ar[r]^-\sim & \fP_n
}$$

It will be  more straightforward to pass  to complements and think of $\fP_n$ as the poset of   proper subsets of $[n]$ containing at least two elements
with the opposite partial order: for $A, A' \subset [n]$, we have $A\geq A'$ if and only if $A\subset A'$.

Thus we will associate a subset   $\varphi(\fp)\subset [n]= \{0,1,  \ldots, n\}$ of at least 
two elements
to each correspondence
$$
\xymatrix{
\fp=(R &  \ar@{->>}[l]_-q  S \ar@{^(->}[r]^-i & T_n)
}
$$
where $i$ is the inclusion of a  subtree, and $q$ is a quotient map of trees.
(The  identity correspondence $\fp_0 \in \fP_{T_n}$ will go to the whole subset $ [n] \subset [n]$.)

It is useful to introduce the $A_{n+2}$-tree $\tilde T_n$ with vertices $V(\tilde T_n) = V(T_n) \cup\{v_{0}, v_{n+1}\}$
and an additional edge connecting $v_0$ and $v_1$ and another connecting $v_{n}$ and $v_{n+1}$.

Following Remark~\ref{rem partition}, we can think of elements of $\fP_{T_n}$ equally well as partitions of $\tilde T_n$ into connected subsets $N_0, F_1, \ldots, F_r, N_n$
with $v_0\in N_0$, $v_n\in N_n$, and $i(S) = F_1 \cup \cdots \cup F_r$.

We will identify the elements of $[n]= \{0,1,  \ldots, n\}$ with the edges of $\tilde T_n$ by matching $i \in [n]$ with the edge connecting $v_i$ and $v_{i+1}$, for all $i=0, \ldots, n$.

Now we define $\varphi$ by including the element $i \in \varphi(\fp) $ if and only if $v_i$ and $v_{i+1}$ are in different parts of the partition of $
\tilde T$  corresponding to $\fp$.
(In particular, we have $\varphi(\fp_0) =  [n] $.)
Note that $ \varphi(\fp) $ has at least two elements since $i(S)$ is nonempty, so that there is at least one part $F_1$ as well as the parts 
$N_0$ and $N_n$.
If  $\fp'$ refines $\fp$ in the sense of Lemma~\ref{lem:partition order} so that $\fp \geq \fp'$ then clearly we have $\varphi(\fp) \subset \varphi(\fp')$.
Finally, any such partition is uniquely determined by the collection of those edges separating its parts.

This concludes the proof of the proposition.
\end{proof}

\section{Arboreal hypersurfaces}


\subsection{Rectilinear version}

By a {\em rooted tree} $\dirT=(T, \rho)$, we will mean  a tree  $T$ equipped with a distinguished vertex $\rho\in V(T)$  called the {\em root vertex}.

The vertices $V(T)$ of a rooted tree naturally form a poset with the root vertex $\rho\in V(T)$ the  unique minimum
and $\alpha< \beta \in V(T)$ if the former is nearer to $\rho$ than the latter. 
To each non-root vertex $\alpha \not = \rho \in V(T)$ there is a unique {\em parent vertex} $\hat \alpha\in V(T)$ such that $\hat\alpha<  \alpha$ and there are no  vertices strictly between them.

Let us write $\R^\dirT = \R^{V(T)}$ for the Euclidean space 
of real tuples
$$
\xymatrix{
\{x_\gamma\},
\text{ with } \gamma\in V(T).
}
$$

\begin{defn}
Fix a rooted tree $\dirT =(T, \rho)$ and a vertex $\alpha\in V(T)$.
 
 (1) Define the quadrant $Q_\alpha \subset \R^\dirT$ to be the closed subspace
$$
Q_\alpha = \{ x_\beta \geq 0 \text{ for all } \beta \leq \alpha\}
$$

(2) Define the hypersurface $H_\alpha \subset \R^\dirT$ to be the boundary
$$
H_\alpha = \partial Q_\alpha
$$
\end{defn}

\begin{remark}\label{rem euclidean space}
Note that the hypersurface $H_\alpha \subset \R^\dirT$ is homeomorphic (in a piecewise linear fashion) to a Euclidean space of dimension $|V(T)| - 1$.
\end{remark}

\begin{defn}
The {\em rectilinear arboreal hypersurface} $H_\dirT$ associated to  a rooted tree $\dirT= (T, \rho)$ is the union of hypersurfaces
$$
H_\dirT =\bigcup_{\alpha\in V(T)}   H_\alpha \subset \R^\dirT
$$
\end{defn}

\begin{remark}
The rectilinear arboreal hypersurface  admits the less redundant presentation 
$$
H_\dirT =\bigcup_{\alpha\in V(T)} \{x_\alpha = 0,  x_\beta > 0 \text{ for all } \beta<  \alpha \}   \subset \R^\dirT
$$
\end{remark}


\subsection{Smoothed version}\label{sect sm arb}

We construct here a smoothed version of the rectilinear arboreal hypersurface of a rooted tree. We  will show in the next section that they are homeomorphic as embedded hypersurfaces inside of Euclidean space.

Fix a continuously differentiable function $b:\R_{>0}\to \R$ with the  properties:
\begin{enumerate}
\item $b$ is non-positive.
\item $\lim_{t\to 0} b(t) = 0$.
\item $\lim_{t\to 0} b'(t) = -\oo$. 
\item $b(t) = 0$, for $t\gg 0$.
\end{enumerate}

Given the above function $b:\R_{>0}\to \R$, choose a continuously differentiable function $f:\R^2\to \R$ with 
the  properties:
\begin{enumerate}
\item $f$ is a submersion.
\item $\{f(x_1, x_2)=0\} = \{x_1 = 0, x_2\geq 0\} \cup \{x_1> 0 , x_2 = b(x_1)\}$.
\item $\{f(x_1, x_2)>0\} =\{x_1>0, x_2>b(x_1)\}$.
\item $\{f(x_1, x_2)<0\} = \{ x_1 < 0\} \cup \{ x_1 = 0, x_2<0\} \cup \{x_1>0,  x_2<b(x_1)\}$.
\end{enumerate}

\begin{remark}
If  preferred, one can fix some $N\geq 1$, and arrange that  $\lim_{t\to 0} b^{(k)}(t) = -\oo$, for all $1\leq k\leq N$. Then one can choose $f$ to be correspondingly
highly differentiable. One can also take $N=\infty$ and then choose $f$ to be smooth. 
\end{remark}


\begin{defn}
Fix a rooted tree $\dirT= (T, \rho)$.

(1) For the root vertex $\rho\in V(T)$,  set
$$
\xymatrix{
h_\rho = x_\rho:\R^{\dirT} \ar[r] & \R
}
$$

(2) For a non-root vertex $\alpha \not = \rho\in V(T)$,  inductively define 
$$
\xymatrix{
h_\alpha :\R^{\dirT} \ar[r] & \R & h_\alpha = f(h_{\hat \alpha}, x_\alpha)
}
$$
where $\hat \alpha \in V(T)$ is the  parent vertex of $\alpha$.
\end{defn}

\begin{remark}\label{rem coord depend}

(1) Note that   $h_\alpha$  depends only on the coordinates $x_\beta$, for  $\beta\leq \alpha$.

(2) Note also that  $h_\alpha  \geq 0$ implies   $ h_{\beta}  \geq 0$, for  $\beta\leq \alpha$. 

%

\end{remark}

\begin{defn}
Fix a rooted tree $\dirT = (T, \rho)$ and a vertex $\alpha\in V(T)$.
 
(1) Define the halfspace $\sQ_\alpha \subset \R^\dirT$ to be the closed subspace
$$
\sQ_\alpha = \{ h_\alpha \geq  0\} 
$$

(2) Define the hypersurface $\sH_\alpha \subset \R^\dirT$ to be the zero-locus
$$
\sH_\alpha = \{ h_\alpha = 0\} 
$$

\end{defn}

%
%
%
%

\begin{defn}

The {\em smoothed arboreal hypersurface} $\sH_\dirT$ associated to  a rooted tree $\dirT= (T, \rho)$ is the union
of hypersurfaces
$$
\sH_\dirT =\bigcup_{\alpha\in V(T)}   \sH_\alpha \subset \R^{\dirT}
$$

\end{defn}

\begin{remark}
The smoothed arboreal hypersurface  admits the less redundant presentation 
$$
\sH_\dirT =\bigcup_{\alpha\in V(T)} \{h_{x_\alpha} = 0,  h_{ x_{\hat\alpha}} > 0 \}   \subset \R^\dirT
$$
where $\hat \alpha \in V(T)$ is the  parent vertex of $\alpha$.
\end{remark}


\subsection{Comparison}
We next compare the rectilinear and smoothed arboreal hypersurfaces.

Given the  function $b:\R_{>0}\to \R$, define the continuous function $\varphi:\R^2\to \R$ by the formula
$$
\varphi(x_1, x_2) = \left\{\begin{array}{cl}
x_2 &  \text{ when } x_1\leq 0\\
x_2-b(x_1) & \text{ when }x_1> 0
\end{array}\right.
$$

\begin{defn}
(1) For the root vertex $\rho\in V(T)$,  set
$$
\xymatrix{
F_\rho:\R^{\dirT} \ar[r] & \R &  F_\rho = x_\rho
}
$$

(2) For a non-root vertex $\alpha \not = \rho\in V(T)$,  set
$$
\xymatrix{
F_\alpha :\R^{\dirT} \ar[r] & \R & F_\alpha = \varphi(h_{\hat \alpha}, x_\alpha)
}
$$
where $\hat \alpha \in V(T)$ is the unique parent of $\alpha$.

(3)
Define the continuous map
$$
\xymatrix{
F_\dirT:\R^{\dirT} \ar[r] & \R^{\dirT} & F_\dirT = \{F_\alpha\}
}
$$
\end{defn}

\begin{remark}
Note that $F_\alpha$  depends only on the coordinates $x_\beta$, for  $\beta\leq \alpha$.

\end{remark}

\begin{prop}\label{prop compare}
The map $F_\dirT:R^\dirT\to \R^\dirT$ is a homeomorphism and restricts to a homeomorphism
from the smoothed to rectilinear arboreal hypersurface
$$
\xymatrix{
F_\dirT: \sH_\dirT \ar[r]^-\sim & H_\dirT
}
$$
satisfying $F_\dirT(\sQ_\alpha) = Q_\alpha$, $F_\dirT(\sH_\alpha) = H_\alpha$, for all $\alpha\in V(T)$.
\end{prop}

\begin{proof}
We will proceed by induction on the size of the vertex set $V(T)$. In the base case when $V(T)$ has a single element, the assertions are evident:
$\sH_\dirT = H_\dirT =  \{0\}\subset \R$, and $F_\dirT:\R\to\R$ is the identity.

Now suppose $V(T)$ contains at least two elements. Let $\tau\in V(T)$ be a maximal vertex in the partial order (in particular $\tau$ will not be the root vertex $\rho$). Introduce the rooted tree $\dirT_\tau = (T_\tau, \rho)$ where we delete the vertex $\tau$ and the edge $\{\tau, \hat \tau\}$ where $\hat\tau\in V(T)$ is
the parent vertex of $\tau$.
Suppose the assertions are already established for $\dirT_\tau$.

Let us first show $F_\dirT:R^\dirT\to \R^\dirT$ is a homeomorphism given that $F_{\dirT_\tau}:R^{\dirT_\tau}\to \R^{\dirT_\tau}$
is a homeomorphism. Under the identification $\R^\dirT = \R^{\dirT_\tau} \times \R^{\{\tau\}}$,  by definition we have
$$
F_\dirT =(F_{\dirT_\tau}, \varphi(h_{\hat \tau}, x_\tau)) 
$$
By construction, we can regard $\varphi:\R^2\to\R$  as a family of homeomorphisms in its second variable depending on its first variable. (In fact, for fixed value of the first variable, the homeomorphism in the second variable is either the identity or a translation.) Since $h_{\hat \tau}$ is independent of the variable $x_\tau$, we see that  $F_\dirT$ is similarly a parameterized  homeomorphism over $\F_{\dirT_\tau}$, and hence itself a homeomorphism.

To finish the proof, it suffices to show for all  $\alpha \in V(T)$ that we have
$$
F_\dirT (\sQ_\alpha) = F_\dirT (\{h_\alpha \geq 0\}) = \{ x_\beta \geq 0 \text{ for all } \beta \leq \alpha\} = Q_\alpha
$$
since passing to  boundaries, we will have $F_\dirT(\sH_\alpha) = H_\alpha$, 
for all  $\alpha \in V(T)$,
and hence passing to the unions of the boundaries $F_\dirT(\sH_\dirT) = H_\dirT$.

Thus we will  show that for all $\alpha\in V(T)$, and $x\in \R^\dirT$, we have
$$
\xymatrix{
h_\alpha(x)\geq 0 
& \iff & 
F_\beta(x) \geq 0, \text{ for all } \beta\leq \alpha
}$$

Recall that for all $\alpha\in V(T_\tau)$, the functions $h_\alpha, F_\alpha$ depend only on the subtree $\dirT_\tau$.
Hence by induction,  for all $\alpha\in V(T_\tau)$, and $x \in \R^{\dirT}$, we have
 $$
\xymatrix{
 h_\alpha(x) \geq 0 
& \iff & 
 F_\beta(x) \geq 0, \text{ for all } \beta\leq \alpha
}$$
Therefore it suffices to show 
$$
\xymatrix{
h_\tau(x)  \geq 0 
& \iff & 
F_\beta(x) \geq   0, \text{ for all } \beta\leq \tau
}$$

Recall that $h_\tau(x)  \geq 0$ implies $h_\beta(x)  \geq 0$, for all $\beta\leq \tau$. Hence by induction,
it suffices to assume  $h_{\hat\tau}(x) \geq   0$, where $\hat \tau \in V(T)$ is the  parent vertex of $\tau$, and  show
$$
\xymatrix{
h_\tau(x)  \geq 0 
& \iff & 
F_\tau(x) \geq   0
}$$

Returning to the definitions,
on the one hand, 
we have  $h_\tau(x) = f(h_{\hat\tau}(x), x_\tau)$. Under the assumption $h_{\hat\tau}(x) \geq   0$, the prescribed properties of $f$ ensure  
$$
\xymatrix{
h_\tau(x) =f(h_{\hat\tau}(x), x_\tau)\geq 0
& \iff & 
x_\tau\geq b(h_{\hat\tau}(x))
}$$
On the other hand, 
we have $F_\tau(x) = \varphi(h_{\hat\tau}(x), x_\tau)$. 
Under the assumption $h_{\hat\tau}(x) \geq   0$, we have $ \varphi(h_{\hat\tau}(x), x_\tau)  = x_\tau - b(h_{\hat\tau}(x))$.
Thus we similarly conclude
$$
\xymatrix{
F_\tau(x) = \varphi(h_{\hat\tau}(x), x_\tau) \geq 0
 & \iff & 
x_\tau - b(h_{\hat\tau}(x))\geq 0
}$$
\end{proof}

\begin{remark}
By scaling the original function $b$ by a positive constant, one  obtains a family of smoothed arboreal hypersurfaces
all compatibly homeomorphic. Moreover, their limit as the scaling constant goes to zero will  be  the  rectilinear arboreal hypersurface.
Thus one can view the smoothed arboreal hypersurface  as a topologically trivial deformation of the rectlinear arboreal hypersurface.
\end{remark}


\subsection{Directed hypersurfaces}\label{sect dir arb}

Let us first review some notions from microlocal geometry.

Let $M$ denote a smooth $n$-dimensional manifold with $\pi:T^*M\to M$ its  cotangent bundle. 
Let  $P^*M$
denote the projectivization of $T^*M$.
  Points of $P^*M$
 are pairs $(x, [v])$ where $x\in M$
and
 $[v] = \R \cdot v \subset T^*_xM$
 is the line through
  $v \not = 0 \in T_x^* M$.
Let 
$S^*M$
denote the spherical projectivization of $T^*M$.
  Points of  $S^*M$ are pairs $(x, [v])$ where $x\in M$
and
 $[v] = \R_{\geq 0}\cdot v \subset T^*_xM$ is the ray through  $v \not = 0 \in T_x^*M$.

Given a submanifold $Y \subset M$, 
let $
T^*_{Y }M \subset T^*M
$
denote its conormal bundle.
Let  $
P^*_{Y }M \subset P^*M
$
denote the projectivized conormal bundle.
 Points of $P^*_{Y}M$ are pairs $(y, [v])$ where $y\in Y$
and
$[v] = \R \cdot v \subset T^*_Y M|_y$ is the line through $v \not = 0 \in T^*_{Y} M|_y$.
Let  $
S^*_{Y }M \subset S^*M
$
denote the spherically projectivized conormal bundle.
Points of $S_Y^*M $ are pairs $(y, [v])$ where $y\in Y$
and
$[v] = \R_{\geq 0}\cdot v \subset T^*_Y M|_y$ is the ray through $v \not = 0 \in T^*_{Y} M|_y$.
Suppose $M$ is equipped with a complete Riemannian metric.
Then the spherical projectivization $S^* M$ is identified with the unit sphere bundle inside of  $T^*M$.

Throughout what follows, by a hypersurface $H\subset M$, we will mean a closed subspace such that $M$ admits a Whitney stratification with $H$
the closure of the $(n-1)$-dimensional stratum.
Thus there exists an open dense  locus $H^{\sm}\subset H$ which is a locally closed $(n-1)$-dimensional differentiable submanifold of $M$. 

We have a natural diagram of maps
$$
\xymatrix{
S^*_{H^{sm}} M \ar[r] & P^*_{H^{sm}} M \ar[r] & H^{sm}
}
$$
where the first  is a two-fold cover and the second is a diffeomorphism.

\begin{defn}
A hypersurface $H\subset M$, with open dense smooth locus $H^{sm}\subset H$, is said to be in {\em good position} if the closure 
$$
\cL^*_H= \ol{P^*_{H^{\mathit{sm}}} M} \subset P^* M,
$$ 
is finite over $H$.
If this holds, we refer to $\cL^*_H$ as 
the {\em coline bundle} of $H$. 
\end{defn}

\begin{remark}
Equivalently, we could require  the closure 
$$
\cR^*_H= \ol{S^*_{H^{\mathit{sm}}} M} \subset S^* M
$$ 
 be finite over $H$.
If this holds, we refer to $\cR^*_H$ as 
the {\em coray bundle} of $H$. 
\end{remark}

\begin{remark}
If $H\subset M$ is in good position, we have a natural diagram of finite maps  
$$
\xymatrix{
\cR^*_H \ar[r] &  \cL^*_H \ar[r] & H
}
$$
where the first is a two-fold cover and the second is a diffeomorphism over $H^{sm}\subset H$.

\end{remark}
\begin{example}
(1) All  Whitney stratified plane curve singularities are in good position.

(2) The  singular  quadratic  cones 
$$
\xymatrix{
\displaystyle
\{\sum_{i=1}^k x_i^2 - \sum_{j=k+1}^n x_j^2 = 0 \}\subset \BR^n
& n> 2, n>k>0 
}
$$
are not in good position.
\end{example}

%

\begin{defn}

(1) 
By a {\em coorientation}  of   a hypersurface $H\subset M$  in good position,
we will  mean   a section 
$$
\xymatrix{
\cR^*_H  \ar[r] &  \cL^*_{H} \ar@/_0.75pc/[l]_-\sigma
}$$
 of the natural two-fold cover from the coray to coline bundle.

(2)
By a  {\em directed hypersurface} $(H, \sigma)$, we will mean a hypersurface $H\subset M$ in good position equipped with a coorientation $\sigma$.

(3) 
By the {\em positive ray bundle} of a directed hypersurface $(H, \sigma)$,
we will mean the image of the  coorientation  
$$
\cR_H^+ = \sigma(\cL^*_H) \subset S^*M
$$
\end{defn}


\medskip

Now let us return to a rooted tree $\dirT=(T, \rho)$ and its smoothed arboreal hypersurface 
$$
\sH_\dirT =\bigcup_{\alpha\in V(T)}   \sH_\alpha \subset \R^{\dirT}
$$

Since $f$ is a submersion, each $h_\alpha:\R^\dirT \to \R$ is a submersion, hence each  hypersurface $\sH_\alpha =\{h_\alpha=0\} \subset \R^\dirT$
is in good position with the natural projection a homeomorphism
$$
\xymatrix{
\pi:\cL^*_{\sH_\alpha} \ar[r]^-\sim & \sH_\alpha
}$$ 
Thus
 the smoothed arboreal hypersurface $\sH_T \subset \R^\dirT$ is in good position (since it is  a finite union of these hypersurfaces).
 
 Moreover,  each  hypersurface $\sH_\alpha  \subset \R^\dirT$
  comes equipped with a preferred coorientation $\sigma_\alpha$
given by the codirection pointing towards the halfspace 
  $\sQ_\alpha =\{h_\alpha\geq 0\} \subset \R^\dirT$. 
  
%

\begin{thm}\label{thm:arbordirect}
Let $\dirT = (T, \rho)$ be a rooted tree 
with  arboreal singularity $\sL_T$
and smoothed arboreal hypersurface $\sH_\dirT \subset \R^\dirT$,

(1) The smoothed arboreal hypersurface $\sH_\dirT \subset \R^\dirT$ admits a natural coorientation $\sigma$ whose restriction to each $\sH_\alpha\subset \sH_\dirT$ is  the  coorientation $\sigma_\alpha$.

(2) Let  $\cL_{\sH_\dirT}^* \subset P^*\R^\dirT$ be the coline bundle of $\sH_{\dirT}\subset \R^\dirT$.
There is a homeomorphism 
$$
\xymatrix{
\varphi:\sL_T\ar[r]^-\sim &  \cL_{\sH_\dirT}^*
 }
 $$
whose   composition with the  natural  projection $\pi: \cL_{\sH_\dirT}^*\to \sH_\dirT$ restricts to homeomorphisms
 $$
 \xymatrix{
\pi\circ \varphi|_{\sL_T(\alpha)}: \sL_T(\alpha)\ar[r]^-\sim &  \sH_\alpha, & 
\text{ for all }  \alpha\in V(T)
}
$$

\end{thm}

\begin{remark}
Note that the composition $\pi\circ \varphi:\sL_T \to \sH_{\dirT}$ will not in general be a homeomorphism but only a finite map. It is possible that distinct points of $\sL_T$ will map to the same point of $\sH_\dirT$ but correspond to different coorientations at that point.
\end{remark}

\begin{proof}[Proof of Theorem~\ref{thm:arbordirect}]
The bulk of the proof will be of assertion (2)

We will proceed by induction on the size of the vertex set $V(T)$. In the base case when $V(T)$ has a single element, all assertions are evident:
$\sH_\dirT = \{0\}\subset \R$, the coorientation points towards the halfspace $\R_{\geq 0}\subset \R$, and $\sL_T  = \{0\}$.

Now suppose $V(T)$ contains at least two elements. Let $\tau\in V(T)$ be a maximal vertex in the partial order. Introduce the rooted tree $\dirT_\tau = (T_\tau, \rho)$ where we delete the vertex $\tau$ and the edge $\{\tau, \hat \tau\}$ where $\hat\tau\in V(T)$ is
the parent vertex of $\tau$.

Suppose the assertions of the theorem are established for $\dirT_\tau$. 
By definition, inside of  
$ \R^\dirT$, 
we have an identification of hypersurfaces
$$
\xymatrix{
\sH_{\dirT} =\left( \sH_{\dirT_\tau} \times\R^{\{\tau\}} \right) \cup \sH_\tau
}
$$
where each factor in the union is a closed subspace.
Therefore 
inside of  $ P^*\R^\dirT$,
we have an identification of subspaces
\begin{equation}\label{eq micro union}
\xymatrix{
\cL^*_{\sH_{\dirT}} =\cL^*_{\sH_{\dirT_\tau} \times\R^{\{\tau\}}} \cup \cL^*_{\sH_\tau}
}
\end{equation}
Note that the first factor in the union admits the presentation
$$
\xymatrix{
\cL^*_{\sH_{\dirT_\tau} \times\R^{\{\tau\}}}\simeq \cL^*_{\sH_{\dirT_\tau}} \times \R^{\{\tau\}}
}
$$

Now let us more closely analyze the  second factor $\cL^*_{\sH_\tau}$ in the union~(\ref{eq micro union}).
By Proposition~\ref{prop compare} and Remark~\ref{rem euclidean space}, there is a homeomorphism
\begin{equation}\label{eq homeo}
\sH_\tau\simeq \R^{|V(T)| - 1}
\end{equation}


Introduce the subspaces
$$
\xymatrix{
\sH^+_\tau = \sH_\tau \cap \{h_{\hat \tau} >0\} &
\sH^0_\tau = \sH_\tau \cap \{h_{\hat \tau} =0\}&
\sH^-_\tau = \sH_\tau \cap \{h_{\hat \tau} <0 \}
& 
}
$$
Recall that $h_{\tau} \geq 0$ implies $h_{\hat\tau} \geq 0$, so that $\sH^-_\tau=\emptyset$ and hence
$$
\xymatrix{
\sH_\tau = \sH_\tau^+ \cup \sH_\tau^0
}$$
Note as well that 
$$
\xymatrix{
\sH^0_\tau \subset \sH_{\hat\tau}\subset  \sH_{\dirT} \times \R^{\{\tau\}}
}$$

Let $\bar \sH^{+}_\tau \subset \sH_\tau$
denote the closure of $\sH^+_\tau$. Then since $\sH^-_\tau=\emptyset$ and $\sH^0_\tau \subset \sH_{\dirT} \times \R^{\{\tau\}}$,  the union~(\ref{eq micro union}) admits the refinement
\begin{equation}\label{eq ref micro union}
\xymatrix{
\cL^*_{\sH_{\dirT}} =\cL^*_{\sH_{\dirT_\tau} \times\R^{\{\tau\}}} \cup \cL^*_{\bar \sH^+_\tau}
}
\end{equation}
where each factor in the union is a closed subspace.

By Proposition~\ref{prop compare}, the  
homeomorphism (\ref{eq homeo}) can be chosen to restrict to a homeomorphism
\begin{equation}\label{eq wing}
\xymatrix{
\bar \sH^+_\tau \simeq    \R^{|V(T)| - 2} \times \R_{\geq 0}
}
\end{equation}
Furthermore, again by Proposition~\ref{prop compare}, under the above identifications, we have
\begin{equation}\label{eq bdy}
\xymatrix{
\partial \bar \sH^+_\tau \simeq  \R^{|V(T)| - 2} \times  \{0\}  \simeq \sH_{\hat \tau} \cap\{x_\tau = 0\} 
}
\end{equation}

Next observe that  projection along the $\tau$-direction is a diffeomorphism
$$
\xymatrix{
\sH^+_\tau =\{h_{\hat \tau} > 0 , h_\tau= 0\} \ar[r]^-\sim & \{h_{\hat\tau} >0, x_\tau = 0\}
}
$$
so that we have the non-intersection of coline bundles
$$
\xymatrix{
\cL^*_{\sH_{\dirT_\tau} \times\R^{\{\tau\}}} \cap \cL^*_{\sH_\tau^+} = \emptyset
}
$$

We conclude by the identifications~(\ref{eq ref micro union}), (\ref{eq wing}), (\ref{eq bdy}), and induction, there is the required homeomorphism 
$$
\xymatrix{
\sL_T\ar[r]^-\sim &  \cL_{\sH_\dirT}^*
 }
 $$
 since we have the presentation
 $$
 \cL^*_{\sH_\dirT}  \simeq ( \cL^*_{\sH_{\dirT_\tau}} \times\R^{\{\tau\}} ) \coprod_{ \R^{|V(T)| - 2} \times \{0\}} ( \R^{|V(T)| - 2} \times
  \R_{\geq 0} )
 $$
  exactly as appears for $\sL_T$ in Lemma~\ref{lem term}.

Finally, to see assertion (1), first note that the coorientation of $\sH_{\dirT_\tau}$ naturally extends to a coorientation of  $\sH_{\dirT_\tau} \times \R^{\{\tau\}}$. Thus by induction, it suffices to return to the union~(\ref{eq micro union}) and check that the coorientations
of $\sH_{\dirT_\tau} \times \R^{\{\tau\}}$ and $\sH_\tau$ agree along their intersection 
$$
\sH_\tau^0 = \{h_\tau  = h_{\hat \tau} = 0\}
$$
Recall that $h_\tau\geq  0$ implies $h_{\hat \tau} \geq 0$, or in other words $\sQ_\tau \subset \sQ_{\hat \tau}$. By definition, the coorientations of $\sH_\tau, \sH_{\hat \tau}\subset \R^\dirT$ point towards the respective halfspace $\sQ_\tau, \sQ_{\hat \tau}\subset \R^\dirT$, hence the coorientations of $\sH_{\dirT_\tau} \times \R^{\{\tau\}}$ and $\sH_\tau$ agree along  
$
\sH_\tau^0.
$
\end{proof}

\section{Microlocal sheaves}


\subsection{Stalk calculation}

Fix a rooted tree $\dirT = (T, \rho)$ where as usual $T$ is a tree and $\rho\in V(T)$ is the root vertex.
Recall that $\R^\dirT$ denotes the Euclidean space of real tuples 
$$
\xymatrix{
\{x_\gamma\},
\text{ with } \gamma\in V(T).
}
$$
and $S^*\R^\dirT$ denotes its spherically projectivized cotangent bundle. The latter is naturally a cooriented contact manifold.

Recall that we have constructed the smoothed arboreal hypersurface 
$$
\sH_\dirT = \bigcup_{\alpha\in V(T)} \sH_\alpha \subset \R^\dirT
$$
It is in good position and comes equipped with a natural coorientation so that its 
positive ray bundle 
is homeomorphic to the arboreal singularity 
$$
\sL_T = \bigcup_{\alpha\in V(T)} \sL_T(\alpha) 
$$
Via this identification, we can regard $\sL_T$ and its Euclidean constituents $\sL_T(\alpha)$ as Legendrian subspaces
 of $S^*\R^\dirT$. When doing so, we will use $\dirT$ in the notation in place of $T$.

\medskip

Fix once and for all a field $k$, and let $\Sh(\R^\dirT)$ denote the dg category of cohomologically constructible complexes of sheaves of $k$-vector spaces on $\R^\dirT$.

Our main object of study will be  the dg category  $\Sh_{\sL_\dirT}(\R^\dirT)$  of constructible complexes of $k$-vector spaces on $\R^\dirT$
microlocalized along the Legendrian subspace
$
\sL_\dirT \subset S^*\R^\dirT.
$
There are two equivalent ways to think about $\Sh_{\sL_\dirT}(\R^\dirT)$ in terms of $\Sh(\R^\dirT)$ which we will now explain.

To any object $\cF\in \Sh(\R^\dirT)$, one can associate its singular support $\ssupp(\cF) \subset S^*\R^\dirT$. This is a closed Legendrian subspace recording those codirections in which the propagation of sections of $\cF$ is obstructed. In particular, one has the vanishing $\ssupp(\cF) = \emptyset$ if and only if the cohomology sheaves of $\cF$ are locally constant. 

\begin{remark}
To fix standard conventions~\cite{KS}, suppose $i:U\to \R^\dirT$ is the inclusion of an open subset with a smooth boundary hypersurface $\partial U \subset \R^\dirT$. Then the singular support of the extension by zero of the constant sheaf $i_! k_U\in \Sh(\R^\dirT)$ consists of the spherical projectivization of the outward conormal codirection along $\partial U \subset \R^\dirT$.
\end{remark}

Abstractly, one can define $\Sh_{\sL_\dirT}(\R^\dirT)$ as the dg quotient category of $\Sh(\R^\dirT)$ by the full subcategory of all objects $\cF$
for which
$\ssupp(\cF) \cap \sL_\dirT = \emptyset$.
Equivalently, one can take $\Sh_{\sL_\dirT}(\R^\dirT)$ to be the full dg  subcategory of $\Sh(\R^\dirT)$ consisting of objects
$\cF$ for which:
\begin{enumerate}
\item $\ssupp(\cF) \subset \sL_\dirT$, and
\item $\Hom_{\Sh(\R^\dirT)}(k_{\R^\dirT}, \cF) \simeq 0$.
\end{enumerate}

We will now give two concrete collections of generators for this subcategory, and one could take their triangulated hulls inside of  $\Sh(\R^\dirT)$ as the definition of $\Sh_{\sL_\dirT}(\R^\dirT)$. 

Recall that for each $\alpha\in V(T)$, the hypersurface $\sH_\alpha \subset \R^\dirT$ is the zero locus of the function 
$$
\xymatrix{
h_\alpha:\R^\dirT\ar[r] &  \R
}$$
Consider the inclusion of the open subspace 
$$
\xymatrix{
i_\alpha:U_\alpha = \{h_\alpha <0\} \ar@{^(->}[r] &  \R^\dirT
}$$
Introduce the extension by zero 
$$\cP_\alpha = i_{\alpha!}k_{U_\alpha} \in  \Sh(\R^\dirT)
$$
Observe the elementary properties:
\begin{enumerate}
\item $\ssupp(\cP_\alpha) = \sL_\dirT(\alpha)$, 
\item $\Hom_{\Sh(\R^\dirT)}(k_{\R^\dirT}, \cP_\alpha) \simeq 0$.
\end{enumerate}

Alternatively,
recall that to each non-root vertex $\alpha \not = \rho \in V(T)$ there is a unique  parent vertex $\hat \alpha\in V(T)$ such that $\alpha> \hat \alpha$ and there are no  vertices strictly between them.
Consider the inclusion of the open subspace 
$$
\xymatrix{
j_\alpha:W_\alpha = \{h_\alpha <0, h_{\hat \alpha} >0\} \ar@{^(->}[r] &  U_\alpha  = \{h_\alpha <0\}
}$$
Introduce the iterated extension 
$$\cS_\alpha = i_{\alpha!} j_{\alpha*} k_{W_\alpha} \in  \Sh(\R^\dirT)
$$
For the root vertex $\rho\in V(T)$, set 
$$
\cS_\rho = \cP_\rho =  i_{\alpha!}k_{U_\alpha} \in  \Sh(\R^\dirT)
$$

Observe the collection of canonical exact triangles
$$
\xymatrix{
\cP_{\hat \alpha} = i_{\hat\alpha!} i_{\hat\alpha}^!\cP_\alpha \ar[r] & \cP_{\alpha} \ar[r] & i_{\alpha!}j_{\alpha*}j_\alpha^*\cP_\alpha = \cS_\alpha \ar[r]^-{[1]} & 
}
$$
With the analogous properties recorded above, the exact triangles   imply the
 properties:
\begin{enumerate}
\item $\ssupp(\cS_\alpha) \subset \sL_\dirT $,
\item $\Hom_{\Sh(\R^\dirT)}(k_{\R^\dirT}, \cS_\alpha) \simeq 0$.
\end{enumerate}

 \begin{remark}
 In fact, the exact triangles imply the precise singular support calculation
$$
\ssupp(\cS_\alpha) = \text{ closure of } (\sL_\dirT(\alpha) \cup \sL_\dirT({\hat \alpha})) \setminus  (\sL_\dirT(\alpha) \cap \sL_\dirT({\hat \alpha}))
$$
 \end{remark}
 
 Furthermore, the exact triangles  also  imply that the triangulated hull of the collection of objects $\cP_\alpha  \in \Sh(\R^\dirT) $, for $\alpha\in V(T)$,
coincides with that of
 the collection of objects  $\cS_\alpha \in \Sh(\R^\dirT)$,  
 for $\alpha\in V(T)$.

\begin{prop}\label{prop gen}
The collection of objects $\cP_\alpha\in \Sh(\R^\dirT)$, for  $\alpha\in V(T)$, or alternatively
 the collection of objects  $\cS_\alpha$,  
 for $\alpha\in V(T)$,
 generates the full dg  subcategory $\Sh_{\sL_\dirT}(\R^\dirT) \subset \Sh(\R^\dirT)$ consisting of objects
$\cF$ for which:
\begin{enumerate}
\item $\ssupp(\cF) \subset \sL_\dirT$, and
\item $\Hom_{\Sh(\R^\dirT)}(k_{\R^\dirT}, \cF) \simeq 0$.
\end{enumerate}

\end{prop}

\begin{proof}
It suffices to prove the assertion for the collection of objects  $\cS_\alpha$,  
 for $\alpha\in V(T)$.
 
 For each $\alpha\in V(T)$, recall the inclusion of the open subspace 
$$
\xymatrix{
i_\alpha:U_\alpha = \{h_\alpha <0\} \ar@{^(->}[r] &  \R^\dirT
}
$$
Introduce the complementary closed inclusion
 $$
\xymatrix{
q_\alpha:\sQ_\alpha = \{h_\alpha \geq 0\} \ar@{^(->}[r] &  \R^\dirT
}
$$
and the open inclusion of its interior
 $$
\xymatrix{
q^\circ_\alpha:\sQ^\circ_\alpha = \{h_\alpha > 0\} \ar@{^(->}[r] &  \R^\dirT
}
$$

Now let us begin with the first step of an iterative procedure to prove the assertion.

Step ($\rho$).
Fix an object $\cF\in\Sh_{\sL_\dirT}(\R^\dirT)$.  To start, we have the canonical exact triangle
 $$
\xymatrix{
i_{\rho !} i_\rho^!\cF \ar[r] & \cF \ar[r] & q_{\rho*}q_\rho^*\cF \ar[r]^-{[1]} & 
}
$$
Note that $i_{\rho !} i_\rho^!\cF$ is a sum of copies of $\cS_\rho$.
Furthermore, the  canonical restriction map
$$
\xymatrix{
 q_{\rho*}q_\rho^*\cF \ar[r] &  q^\circ_{\rho*}q_\rho^{\circ*}\cF 
}
$$
 is a quasi-isomorphism
since its cone is an object supported within $\sH_\rho$ but with singular support inside the codirection $\sL_\dirT(\rho)$  and hence must vanish.

Set $\cF_{>\rho} = q_\rho^{\circ*}\cF \in \Sh(\sQ^\circ_\rho).$


Consider the descendant set $d(\rho) \subset V(T)$ of those vertices $\alpha\in V(T)$ with parent vertex $\hat\alpha = \rho \in V(T)$. Note that the inclusion of $d(\rho)$ into the disjoint union
of trees  $T\setminus \{\rho\}$ gives a bijection with its connected components.
Recall by construction that $\sL_\dirT(\rho)$ is the union of $\sL_\dirT(\fp)$ for those partitions
$$
\xymatrix{
\fp=(R &  \ar@{->>}[l]_-q  S \ar@{^(->}[r]^-i & T)
}
$$
such that $\rho \in V(S)$.
The poset of those partitions $\fp$ such that $\rho \not\in V(S)$ is a disjoint union indexed by $d(\rho)$ given by which component of $T\setminus \{\rho\}$ contains $S$.
Thus by Theorem~\ref{thm reg cc}, the complement
$\sL_\dirT \setminus \sL_\dirT(\rho)$ is
a disjoint union indexed by $d(\rho)$, and
 hence $\cF_{>\rho} $ is a corresponding direct sum indexed by $d(\rho)$. 

For each $\alpha\in d(\rho)$, set $\cF_{\geq \alpha} \subset \cF_{>\rho} $ to be the corresponding summand.

This concludes Step ($\rho$).

Step ($\alpha$), for each $\alpha\in d(\rho)$.

Now  repeat  the  recollement pattern  introduced above at Step ($\rho$), but applied to each summand  $\cF_{\geq \alpha}$
corresponding to the root vertex $\alpha\in d(\rho)$ of a connected component of the remaining disjoint union of trees $T\setminus\{\rho\}$. It results in the expression of each $\cF_{\geq \alpha}$  as an extension built
out of sums of copies of  $\cS_\alpha$ and a new sheaf $\cF_{> \alpha}$ which is itself a sum
of sheaves $\cF_{\geq \alpha'}$, for $\alpha'\in d(\alpha)$, to which we can then apply the next iterative step.
Here as before we write $d(\alpha) \subset V(T)$ for the descendant set of those vertices $\alpha'\in V(T)$ with parent vertex $\hat\alpha' = \alpha \in V(T)$.

Let us spell out the general Step ($\alpha$) where we start with $\cF_{\geq \alpha} \subset \cF_{>\hat\alpha} $ but do not assume that $\alpha \in d(\rho)$. We have the canonical exact triangle
 $$
\xymatrix{
i_{\alpha !} i_\alpha^!q^\circ_{\hat\alpha*} \cF_{\geq \alpha}  \ar[r] & q^\circ_{\hat\alpha*}\cF_{\geq \alpha}  \ar[r] & q_{\alpha*}q_\alpha^*\cF_{\geq \alpha}  \ar[r]^-{[1]} & 
}
$$
Note that $i_{\alpha !} i_\alpha^!q^\circ_{\hat\alpha*} \cF_{\geq \alpha}$  is a sum of copies of $\cS_\alpha$.
Furthermore, the  canonical restriction map
$$
\xymatrix{
 q_{\alpha*}q_\alpha^*\cF_{\geq \alpha} \ar[r] &  q^\circ_{\alpha*}q_\alpha^{\circ*}\cF_{\geq \alpha}
}
$$
 is a quasi-isomorphism
since its cone is an object supported within $\sH_\alpha$ but with singular support inside the codirection $\sL_\dirT(\alpha)$  and hence must vanish.

Set $\cF_{>\alpha} = q_\alpha^{\circ*}\cF_{\geq \alpha} \in \Sh(\sQ^\circ_\alpha).$

Consider the descendant set $d(\alpha) \subset V(T)$ of those vertices $\alpha'\in V(T)$ with parent vertex $\hat\alpha' = \alpha \in V(T)$. 
By Theorem~\ref{thm reg cc}, observe that $\sL_\dirT \setminus (\cup_{\gamma\not > \alpha} \sL_\dirT(\gamma))$ is a disjoint union indexed by $d(\alpha)$, and hence $\cF_{>\alpha} $ is a corresponding direct sum indexed by $d(\alpha)$. 

For each $\alpha'\in d(\alpha)$, set $\cF_{\geq \alpha'} \subset \cF_{>\alpha} $ to be the corresponding summand.

Now continue inductively vertex by vertex following the partial order.
\end{proof}

%

Now we will calculate the  dg category  $\Sh_{\sL_\dirT}(\R^\dirT)$ by calculating the morphisms between 
the generating objects $\cP_\alpha\in \Sh(\R^\dirT)$, for  $\alpha\in V(T)$.

Recall that we can regard  the rooted tree $\dirT = (T, \rho)$ as a poset with the root vertex $\rho\in V(T)$
the unique minimum. 
To each non-root vertex $\alpha \not = \rho \in V(T)$ there is a unique parent vertex $\hat \alpha\in V(T)$ such that $\alpha> \hat \alpha$ and there are no  vertices strictly between them.

Now let us regard the rooted tree $\dirT= (T, \rho)$ as a quiver with a unique arrow pointing from each non-root vertex $\alpha \not = \rho \in V(T)$ to its parent vertex
$\hat \alpha\in V(T)$. Symbolically, we replace the relation $\alpha >\hat\alpha$ with the relation $\alpha \to \hat \alpha$.

 Let $ \Mod(\dirT)$ denote the dg derived category of finite-dimensional complexes of modules over  $\dirT$ regarded as a quiver.
 Objects assign to each vertex $\alpha \in V(T)$ a finite-dimensional complex of $k$-vector spaces $M(\alpha)$, and to each arrow $\alpha \to \hat \alpha$ a degree zero chain map $m_\alpha:M(\alpha) \to M({\hat \alpha})$.
 
 Let us point out two natural generating collections for $\Mod(\dirT)$. There are the simple modules $S_\alpha \in \Mod(\dirT)$ that assign 
 $$
S_\alpha(\beta) = \left\{\begin{array}{cl}
k  &  \text{ when } \beta = \alpha\\
0 & \text{ when } \beta\not =\alpha
\end{array}\right.
$$
with all maps $m_\beta:S_\alpha(\beta) \to S_\alpha(\hat \beta)$ necessarily zero.
 There are also the projective modules $P_\alpha \in \Mod(\dirT)$ that assign 
 $$
P_\alpha(\beta) = \left\{\begin{array}{cl}
k  &  \text{ when } \beta \leq \alpha\\
0 & \text{ when } \beta > \alpha
\end{array}\right.
$$ 
with the  maps $m_\beta:P_\alpha(\beta) \to P_\alpha(\hat \beta)$ the identity isomorphism whenever both domain and range are nonzero.

\begin{thm}\label{thm equiv}
There is a canonical equivalence
$$
\xymatrix{
\varphi:\Sh_{\sL_\dirT}(\R^\dirT) \ar[r]^-\sim &  \Mod(\dirT)
}
$$
such that
$
\varphi(\cP_\alpha) = P_\alpha$ and $\varphi(\cS_\alpha) = S_\alpha $, for all $\alpha\in V(T)$.
\end{thm}

\begin{proof}
It suffices to establish the following:
$$
\xymatrix{
\Hom(\cP_\alpha, \cP_\beta) \simeq 0, &
\text{ when } \alpha\not\leq \beta\in V(T) 
}$$
$$
\xymatrix{
\Hom(\cP_\alpha, \cP_\beta) \simeq k\cdot e_{\alpha}^{\beta}, &
\text{ when } \alpha\leq \beta\in V(T) 
}$$
where $e_\alpha^\beta$ is a generator of degree zero satisfying
$$
\xymatrix{
e_\alpha^\gamma = e_\beta^\gamma \circ e_{\alpha}^{\beta}, & 
\text{ when } \alpha\leq \beta\leq \gamma\in V(T) 
}$$

To start, for any $\alpha, \beta\in V(T)$, we have
$$
\Hom(\cP_\alpha, \cP_\beta) =\Hom(i_{\alpha!}k_{U_\alpha}, i_{\beta!}k_{U_\beta}) 
\simeq
\Hom(k_{U_\alpha}, i_\alpha^!i_{\beta!}k_{U_\beta}) 
\simeq
C^*(U_\alpha\cap \ol{U}_\beta, U_\alpha\cap \partial U_\beta; k)
$$
where the last term is the complex of relative singular cochains.
Furthermore, the composition of morphisms is the natural cup product 
$$
\xymatrix{
\cup:C^*(U_\alpha\cap \ol{U}_\beta,U_\alpha\cap \partial U_\beta; k)\otimes
C^*(U_\beta\cap \ol{U}_\gamma, U_\beta\cap\partial U_\gamma; k)\ar[r] &
C^*(U_\alpha\cap \ol{U}_\gamma, U_\alpha\cap\partial U_\gamma; k)
}
$$
induced by viewing
relative singular cochains as a subcomplex of singular cochains and taking
 the restriction of the usual cup product.

When $\alpha\not\leq \beta\in V(T)$, 
then either of two cases hold: (i)  $\alpha>\beta $, or (ii) $\alpha$ and $\beta$ are not comparable.
Let us  verify in each  case the relative cohomology vanishes
$$
C^*(U_\alpha\cap \ol{U}_\beta, U_\alpha\cap \partial U_\beta; k)\simeq 0
$$ 

We will appeal
 to Proposition~\ref{prop compare}  in order to assume the rectilinear presentation
$$
\xymatrix{
U_\alpha \simeq \{x_\gamma < 0 \text{ for some } \gamma\leq \alpha\}
&
U_\beta \simeq \{x_\gamma < 0 \text{ for some } \gamma\leq \beta\}
}
$$
so that we  also have
$$
\xymatrix{
\partial U_\beta \simeq \{x_\gamma \geq  0 \text{ for all } \gamma\leq \beta  ;\,  x_{\gamma} = 0 \text{ for some } \gamma \leq \beta\}
}
$$

Since $U_\beta$ is homeomorphic to an open halfspace with $\partial U_\beta$ homeomorphic to a hyperplane, it suffices to see that $U_\alpha\cap \partial U_\beta$ is contractible.

Therefore in case (i) when $\alpha>\beta$,  we have
$$
\xymatrix{
U_\alpha\cap \partial U_\beta \simeq \{x_\gamma < 0 \text{ for some } \beta<\gamma\leq \alpha  ;\, 
x_\gamma \geq  0 \text{ for all } \gamma\leq \beta   ;\,  x_\gamma = 0 \text{ for some } \gamma \leq \beta \}
}
$$
which is clearly contractible: contract along straight lines taking $x_\gamma\to -1$, for all $\beta<\gamma \leq \alpha$,
and 
$x_\gamma\to 0$, else.

In case (ii) when  $\alpha$ and $\beta$ are not comparable, let $\eta\in V(T)$ be the maximal element such that $\eta \leq \alpha, \beta.$ Then we have
$$
\xymatrix{
U_\alpha\cap \partial U_\beta \simeq \{x_\gamma < 0 \text{ for some } \eta<\gamma\leq \alpha  ;\, 
x_\gamma \geq  0 \text{ for all } \gamma\leq \beta   ;\,  x_\gamma = 0 \text{ for some } \gamma \leq \beta \}
}
$$
which is also clearly contractible:
contract along straight lines taking $x_\gamma\to -1$, for all $\eta<\gamma \leq \alpha$,
and 
$x_\gamma\to 0$, else.


Now when $\alpha\leq \beta\in V(T)$, we have $U_\alpha\subset U_\beta$ with $U_\alpha$ contractible, hence the morphism complex simplifies to
$$
\xymatrix{
\Hom(\cP_\alpha, \cP_\beta) \simeq C^*(U_\alpha; k) \simeq k\cdot e_\alpha^\beta
}$$ 
where $e_\alpha^\beta$ denotes the constant cochain of degree zero and value $1\in k$.
Furthermore, for $\alpha\leq \beta\leq \gamma\in V(T)$, the composition 
$$
\xymatrix{
\Hom(\cP_\alpha, \cP_\beta) \otimes
\Hom(\cP_\beta, \cP_\gamma) \ar[r] &
\Hom(\cP_\alpha, \cP_\gamma) 
}
$$
simplifies to the natural cup product of cochains
$$
\xymatrix{
\cup:C^*(U_\alpha; k)\otimes
C^*(U_\beta; k)\ar[r] &
C^*(U_\alpha; k)
}
$$
which clearly satisfies $e_\alpha^\beta \cup e_\beta^\gamma = e_\alpha^\gamma$.
\end{proof}


\subsection{Restriction functors}
We continue with a fixed rooted tree $\dirT= (T, \rho)$ with smoothed arboreal hypersurface 
$$
\sH_\dirT = \bigcup_{\alpha\in V(T)} \sH_\alpha \subset \R^\dirT
$$
with conormal Legendrian the arboreal singularity 
$$
\sL_\dirT = \bigcup_{\alpha\in V(T)} \sL_\dirT(\alpha) \subset S^*\R^\dirT
$$

Recall that the  strata $\sL_\dirT(\fp) \subset \sL_\dirT$ are contractible and
 indexed by correspondences
$$
\xymatrix{
\fp=(R &  \ar@{->>}[l]_-q  S \ar@{^(->}[r]^-i & T) 
}
$$
where $i$ is the inclusion of  a subtree, and $q$ is a quotient map of trees.
Furthermore, the normal slice to the stratum  
is homeomorphic to the arboreal singularity $\sL_R$.
%
%

Fix any point $\lambda\in \sL_\dirT(\fp)  $ with projection $x = \pi(\lambda) \in \sH_\dirT$. Choose a small open ball $B(\fp)\subset \R^\dirT$
around $x$. The open set $\pi^{-}(B) \subset S^*\R^\dirT$ intersects $\sL_\dirT$ in possibly many connected components. Let $\Lambda(\fp)\subset \sL_\dirT$ denote the connected component containing $\lambda$.

Introduce the dg category  $\Sh_{\Lambda(\fp)}(B(\fp))$ of constructible complexes of $k$-vector spaces on $B(\fp)$ microlocalized
along the Legendrian subspace $\Lambda(\fp) \subset S^* B(\fp)$.
Restriction of sheaves along the open inclusion  $B(\fp)\subset \R^\dirT$ induces  a natural microlocal restriction functor 
$$
\xymatrix{
\res:\Sh_{\sL_\dirT}(\R^\dirT) \ar[r] & \Sh_{\Lambda(\fp)}(B(\fp))
}
$$

Let us denote by $\cN(\fp) \subset Sh_{\sL_\dirT}(\R^\dirT)$ the full dg subcategory generated by the objects:
\begin{enumerate}
\item
$\cP_\alpha\in \Sh_{\sL_\dirT} (\R^\dirT)$,  when  $\alpha \not\in i(V(S))$,
\item
$\cS_\alpha \in  \Sh_{\sL_\dirT} (\R^\dirT)$,  when  $\alpha, \hat\alpha \in i(V(S)), q(\alpha) = q(\hat\alpha) \in V(R)$.
\end{enumerate}

Observe that the singular support  of any of the above  generating objects, and hence any object  of $\cN(\fp)$,  is disjoint from   $\sL_\dirT(\fp) \subset \sL_\dirT$.
Thus we have the evident vanishing:
$$
\xymatrix{
\res(\cF) \simeq 0, & \text{ for any } \cF\in \cN(\fp).
}$$

\begin{remark}
Thanks to the canonical exact triangle
$$
\xymatrix{
\cP_{\hat \alpha} = i_{\hat\alpha!} i_{\hat\alpha}^!\cP_\alpha \ar[r]^-u & \cP_{\alpha} \ar[r] & i_{\alpha!}j_{\alpha*}j_\alpha^*\cP_\alpha = \cS_\alpha \ar[r]^-{[1]} & 
}
$$
the  vanishing $\res(\cS_\alpha) \simeq 0$ is equivalent to $\res(u)$ being a quasi-isomorphism.
\end{remark}

We will see that the  microlocal restriction  functor 
exhibits $ \Sh_{\Lambda(\fp)}(B(\fp))$ as the dg quotient of
$Sh_{\sL_\dirT}(\R^\dirT)$ by the dg subcategory $\cN(\fp)$.

To spell this out, observe that the quiver structure on $\dirT$ induces one on the subtree $S$ and subsequent quotient tree $R$.
Let us write $\mathcal S$ and $\mathcal R$ to denote $S$ and $R$ with their respective quiver structures.

 Consider the inclusion of the subtree
$
i:S\hra T.
$
Define the quotient functor 
$$
\xymatrix{
i^*:\Mod(\dirT) \ar[r] &  \Mod(\mathcal S)
}
$$ by killing  the projective modules $P_\alpha \in \Mod(\dirT)$, when $\alpha\not \in i(V(S))$.

Consider the quotient map $q:S\too R$.
Define the quotient functor 
$$
\xymatrix{
q_!:\Mod(\mathcal S) \ar[r] & \Mod(\mathcal R)
}
$$ 
by killing  the simple modules $S_\alpha \in \Mod(\mathcal S)$, when $q(\alpha) = q(\hat \alpha) \in V(R)$.

Observe that the composite
$$
\xymatrix{
q_!i^*: Mod(\dirT)\ar[r] & Mod(\mathcal R)
}
$$
is the quotient functor by the full dg subcategory $N(\fp) \subset \Mod(\dirT)$ generated by 
\begin{enumerate}
\item
$P_\alpha\in \Mod(\dirT)$,  when  $\alpha \not\in i(V(S))$,
\item
$S_\alpha \in  \Mod(\dirT)$,  when  $\alpha, \hat\alpha \in i(V(S)), q(\alpha) = q(\hat\alpha) \in V(R)$.
\end{enumerate}

\begin{prop}
There is a natural commutative diagram 

$$
\xymatrix{
\ar[d]_-\res \Sh_{\sL_\dirT}(\R^\dirT) \ar[r]^-\sim &  \Mod(\dirT) \ar[d]^-{q_!  i^*}\\
 \Sh_{\Lambda(\fp)}(B(\fp)) \ar[r]^-\sim &  \Mod(\mathcal R) \\
}
$$
\end{prop}

\begin{proof}
The proof is essentially a repeat of the proofs of Proposition~\ref{prop gen} and Theorem~\ref{thm equiv}.

Choose a vertex $\alpha\in V(T)$ in each fiber of $q:S\too R$, and denote their union by $\tilde V(R) \subset V(T)$. Consider the 
collection of objects $\cP_\alpha\in Sh(\R^\dirT)$, for $\alpha\in \tilde V(R)$. Denote by 
$\cP_{\ol \alpha}\in Sh(B(\fp))$, for $\alpha\in \tilde V(R)$
 their restrictions along the open inclusion
$B(\fp) \subset \Sh(\R^\dirT)$.

By the same argument as in the proof of Proposition~\ref{prop gen}, we see that the collection of objects
$\cP_{\ol \alpha}\in Sh(B(\fp))$, for $\alpha\in \tilde V(R)$ generates  $ \Sh_{\Lambda(\fp)}(B(\fp))$.

By the same argument as in the proof of Theorem~\ref{thm equiv}, we see that the generating objects
$\cP_{\ol \alpha}\in Sh(B(\fp))$, for $\alpha\in \tilde V(R)$, give an equivalence  
$ \Sh_{\Lambda(\fp)}(B(\fp))\simeq \Mod(\mathcal R)$.
\end{proof}





\begin{thebibliography}{99}

\bibitem{ad}
F. Ardila and M. Develin, Tropical hyperplane arrangements and oriented matroids, Math. Zeit., Volume 262, Issue  4 (2009), pp. 795--816.
%


\bibitem{BNP}
D. Ben-Zvi, D. Nadler, and A. Preygel, Integral transforms for coherent sheaves, 
arXiv:1312.7164.


\bibitem{bernstein}
J. Bernstein, Algebraic theory of $\cD$-modules, available online.


\bibitem{bgp}
 I. N. Bernstein, I. M. Gelfand, and V. A. Ponomarev, Coxeter functors and Gabriel's theorem,
Uspehi Mat. Nauk 28 (1973), no. 2 (170), 19--33;
 Russian Math. Surveys 28 (1973), 17--32.
 
\bibitem{bb}
L.J. Billera and A. Bj\"orner, Face numbers of polytopes and complexes, in Handbook of Discete and Computational Geometry, CRC Press, 1997.

\bibitem{bjorner}
A. Bj\"orner, Posets, regular CW complexes and Bruhat order, Europ. J. Combin. 5 (1984), 7--16.


\bibitem{ds}
M. Develin and B. Sturmfels, Tropical convexity, Doc. Math. 9 (2004), 1--27.

\bibitem{DK} T. Dyckerhoff and M. Kapranov, Higher Segal Spaces I,
arXiv:1212.3563.



\bibitem{DKcyc} T. Dyckerhoff and M. Kapranov, Triangulated surfaces in triangulated categories,
arXiv:1306.2545.

%
%
\bibitem{KS}
M. Kashiwara and P. Schapira, 
{\sl Sheaves on manifolds}. Grundlehren der Mathematischen Wissenschaften 292, Springer-Verlag (1994).
%


\bibitem{KellerICM} B. Keller, On differential graded categories,
International Congress of Mathematicians.
Vol. II, 151--190, Eur. Math. Soc., Z\"urich (2006).



%
\bibitem{kont} M. Kontsevich, Symplectic Geometry of Homological Algebra, lecture
notes:\begin{verbatim}
http://www.ihes.fr/~maxim/TEXTS/Symplectic_AT2009.pdf
\end{verbatim}


\bibitem{loday}
J.-L. Loday, 
The Multiple facets of the associahedron, Clay Mathematics Institute Publication, 2005.

%

\bibitem{mccrory}
C. McCrory, Cone complexes and PL transversality, Trans. of the AMS, 207, 1975, 269--291.


\bibitem{Ncyc} D. Nadler, Cyclic symmetries of $A_n$-quiver representations, arXiv:1306.0070.


\bibitem{Nexp} D. Nadler, Non-characteristic expansions of Legendrian singularities, arXiv:1507.01513.


\bibitem{N3dlg} D. Nadler, A combinatorial calculation of the Landau-Ginzburg model $M=C^3, W=z_1 z_2 z_3$,
arXiv:1507.08735.


\bibitem{Nlg} D. Nadler, Mirror symmetry for the Landau-Ginzburg A-model $M=\mathbb C^n, W=z_1 ... z_n$, 
arXiv:1601.02977.


\bibitem{Nwms} D. Nadler, Wrapped microlocal sheaves, in preparation.



\bibitem{NRSSZ}
L. Ng, D. Rutherford, V. Shende, S. Sivek, and E. Zaslow,
Augmentations are Sheaves,
 arXiv:1502.04939.



 
\bibitem{Seidel} P. Seidel, {\sl Fukaya Categories and Picard-Lefschetz Theory,}
 Zurich Lectures in Advanced Mathematics. European Mathematical Society (EMS), Z\"urich, 2008.


\bibitem{STZ}
V. Shende, D. Treumann, and E. Zaslow,
Legendrian knots and constructible sheaves,
arXiv:1402.0490.

\bibitem{ss}
D. Speyer and B. Sturmfels, Tropical Mathematics, Clay Mathematics Institute Senior Scholar Lecture given at
Park City, Utah, July 2004, math.CO/0408099.

\bibitem{Stasheff}
J.D. Stasheff, Homotopy Associativity of H-Spaces. I, II, Trans. AMS, 108 (2), 1963, 275--312.
%

%

\bibitem{wachs}
M.L. Wachs, Poset topology: tools and applications, Geometric Combinatorics, IAS/PCMI lecture notes series (E. Miller, V. Reiner, B. Sturmfels, eds.), 13 (2007), 497--615.


\end{thebibliography}
\end{document}